\documentclass[10pt]{elsarticle}
\usepackage{standalone}

\usepackage{datetime}
\usepackage{lmodern}

\usepackage{graphicx}

\newif\ifproofs

\proofstrue

\usepackage[paper=a4paper,bottom=2.5cm,left=2.5cm,right=2.5cm,top=2.5cm]{geometry}
\usepackage{amsmath,amsfonts,amscd,amssymb,dsfont} %
\usepackage{cleveref}
\usepackage{bm}
\usepackage{textcomp}

\usepackage[shortlabels]{enumitem}

\usepackage{graphics} %

\usepackage{amsthm}

\usepackage{pgfplots}
\pgfplotsset{compat=1.11}
\usepgfplotslibrary{fillbetween}
\usetikzlibrary{intersections}
\pgfdeclarelayer{bg}
\pgfsetlayers{bg,main}

\newtheorem{theorem}{Theorem}[section]

\newtheorem{proposition}{Proposition}[section]
\newtheorem{lemma}{Lemma}[section]

\newtheorem{assumption}{Assumption}

\theoremstyle{definition}
\newtheorem{definition}{Definition}[section]
\newtheorem*{notation}{Notations}
\newtheorem{remark}{Remark}[section]

\Crefname{corollary}{Corollary}{Corollaries}
\Crefname{equation}{Eq.}{Eqs.}
\Crefname{figure}{Figure}{Figures}
\Crefname{tabular}{Tab.}{Tabs.}
\Crefname{table}{Tab.}{Tabs.}
\Crefname{lemma}{Lemma}{Lemmas}
\Crefname{theorem}{Theorem}{Theorems}
\Crefname{definition}{Definition}{Definitions}
\Crefname{section}{Section}{Sections}
\Crefname{proposition}{Proposition}{Propositions}
\Crefname{assumption}{Assumption}{Assumptions}
\Crefname{example}{Example}{Examples}

\usepackage[disable,colorinlistoftodos,bordercolor=orange,backgroundcolor=orange!20,linecolor=orange,textsize=scriptsize]{todonotes}
\usepackage{multirow}
\usepackage{pbox}
\usepackage{amsfonts}

\newcommand{\norm}[1]{\left\Vert #1\right\Vert}
\newcommand{\tnorm}[1]{\textstyle\left\Vert #1\right\Vert}

\newcommand{\txt}{\textstyle}
\newcommand{\mr}{\mathrm}

\newcommand{\x}{x}
\newcommand{\xx}{\bm{x}}
\newcommand{\y}{y}
\newcommand{\yy}{\bm{y}}

\newcommand{\zz}{\bm{z}}

\newcommand{\hf}{\hat{f}}
\newcommand{\g}{\mathbf{g}}
\newcommand{\gxx}{\mathbf{g}_{\xx}}
\newcommand{\gyy}{\mathbf{g}_{\yy}}
\newcommand{\bh}{\mathbf{h}}

\newcommand{\xag}{X}
\newcommand{\xxag}{\bm{\xag}}
\newcommand{\yag}{Y}
\newcommand{\yyag}{\bm{\yag}}
\newcommand{\zag}{Z}
\newcommand{\zzag}{\bm{\zag}}

\newcommand{\rit}{\mathbb{R}}
\newcommand{\nit}{\mathbb{N}}

\newcommand{\X}{\mathcal{X}}
\newcommand{\FX}{\bm{\X}} %
\newcommand{\FXS}{\FX_S}
\newcommand{\Sxag}{\overline{\X}} %
\newcommand{\Y}{\mathcal{Y}}

\newcommand{\I}{\mathcal{I}}

\newcommand{\GA}{\mathcal{G}(A)}

\newcommand{\M}{\mathcal{M}}

\newcommand{\Gna}{G} %
\newcommand{\GnaA}{G(A)}

\newcommand{\ww}{\bm{w}} 
\newcommand{\wwag}{\bm{W}} 
 
\newcommand{\cc}{\bm{c}} %

\newcommand{\diamX}{R} %

\newcommand{\Bdf}{L_\mathbf{f}}
\renewcommand{\th}{\theta} %
\renewcommand{\t}{t}
\renewcommand{\i}{i}

\newcommand{\thti}{_{\th,\t}}
\def\ub{\underline{b}}
\def\ob{\overline{b}}

\newcommand{\hxx}{\hat{\xx}}

\newcommand{\hxxag}{\hat{\xxag}}

\newcommand{\hz}{\hat{\zz}}

\newcommand{\hzz}{\bm{\hat{z}}}

\newcommand{\sxx}{\xx^*}  %
\newcommand{\sxxag}{\xxag^*}
\newcommand{\syy}{{\yy^*}} 

\newcommand{\sz}{{\zz^*}}

\newcommand{\eqd}{\triangleq}
\newcommand{\dth}{\,\mathrm{d}\th}

\newcommand{\nn}{\bm{n}}
\newcommand{\snu}{\nu} %
\newcommand{\esnu}{^{\snu}}

\newcommand{\dset}{\delta} %
\newcommand{\mdset}{\overline{\delta}} %

\newcommand{\duti}{d} %
\newcommand{\mduti}{\overline{d}} %

\newcommand{\stgccvut}{{\alpha}} %

\newcommand{\disth}{\sigma} %
\newcommand{\cutp}{\upsilon} %

\renewcommand{\ss}{\bm{s}}

\newcommand{\A}{\bm{P}}
\newcommand{\bb}{\bm{b}}
\newcommand{\us}{\underline{s}}
\newcommand{\os}{\overline{s}}
\newcommand{\dimp}{K} %
\renewcommand{\L}{\Lambda}
\newcommand{\mld}{\overline{\lambda}}
\newcommand{\ld}{\lambda}

\newcommand{\rlt}{\operatorname{ri}\,}
\newcommand{\rbd}{\operatorname{rbd}\,}
\newcommand{\spa}{\operatorname{span}\,}
\newcommand{\aff}{\operatorname{aff}\,}
\newcommand{\supk}{^{(k)}}
\newcommand{\bpsi}{{\psi}} %

\newcommand{\rhoz}{\rho^*} %
\newcommand{\rhoZ}{\bar{\rho}} %

\newcommand{\rhomin}{\rho} %

\newcommand{\Symp}{ {L}_\mr{S}} %

\renewcommand{\H}{\mathcal{H}} %
\newcommand{\D}{\mathcal{D}} %

\newcommand{\deriv}{\nabla} %

\journal{Journal of Mathematical Analysis and Applications}

\begin{document}

\begin{frontmatter}

\title{%
Nonatomic Aggregative Games with Infinitely Many Types
\tnoteref{t1}}

\tnotetext[t1]{This work was partially supported  by the PGMO program of \textit{Fondation Math\'ematiques Jacques Hadamard}. }

\author{Paulin Jacquot \corref{cor1}%
\fnref{fn1}}

\ead{paulin.jacquot@polytechnique.edu}

\author{Cheng Wan \fnref{fn2}}

\ead{cheng.wan@edf.fr}

\cortext[cor1]{Corresponding author}

\fntext[fn1]{Paulin Jacquot is with  EDF R\&D OSIRIS, Inria and  Ecole polytechnique, CNRS, Palaiseau, France.}

\fntext[fn2]{Cheng Wan is with EDF R\&D  OSIRIS, Palaiseau, France.}

\begin{abstract}
 We define  and analyze the notion of variational Wardrop equilibrium for nonatomic aggregative games with an infinity of players types. 
 These equilibria are characterized through an infinite-dimensional variational inequality. %
  We show, under monotonicity conditions,  a convergence theorem enables to approximate such an equilibrium with arbitrary precision.
To this end, we introduce  a sequence of nonatomic games with a finite number of players types, which approximates the initial game. We show the existence of a symmetric Wardrop equilibrium in each of these games.
     We prove that those symmetric equilibria converge to an equilibrium of the infinite game, and that they can be computed as solutions of finite-dimensional  variational inequalities. 
 The model is illustrated through an example from smart grids:  the description of a large population of electricity consumers by a parametric distribution gives a nonatomic game with an infinity of different players types, 
with actions subject to coupling constraints.
\end{abstract}


\end{frontmatter}

\section{Introduction}
We study the existence and uniqueness of variational Wardrop equilibrium (VWE) in nonatomic aggregative games with coupling aggregative constraints, where a continuum of players have heterogeneous compact convex pure-action sets and cost functions. 
We establish the convergence, to the VWE of such a game, of a sequence of symmetric VWE in auxiliary games with a finite number of types of players.
\medskip

\paragraph{A motivating example} Consider the example of an energy operator studying the flexibility potential between peak and off-peak periods in a large population of energy consumers, for instance all households in France.
 
The operator considers that each household $i$ has a certain quantity of energy $E_i$ that can be balanced between consumption on peak period $\x_{P,i}$ and consumption on off-peak periods $\x_{O,i}$, such that $\x_{O,i}+ \x_{P,i}= E_i$, depending on the cost (per unit of energy) $c_P(\xag_P)$ and $c_O(\xag_O) $ associated with the peak and off peak periods. 
The total on-peak consumption  $X_P= \sum_i \x_{P,i}$ and off-peak consumption $X_O=\sum_i \x_{O,i}$ affect the prices on the energy market and, therefore, change the costs $c_P(X_P)$ and $c_O(X_O)$ set by the operator. 

  The operator wants to compute an equilibrium of this game (for instance to design tariffs). 
  For practical and privacy reasons, it is impossible to have access to the flexibility potential $E_i$ of the thirty millions of French households. 
  However, the operator may have an easier access to a precise parametric, continuous  distribution function of the flexibility potential among the French households.
  
  Then, using the inverse transform sampling method, the game is replicated by modeling the population of households as a continuum $\Theta=[0,1]$ and associating to each $\th \in \Theta$ the flexible energy quantities $E_\th=F_E^{-1}(\th)$ from the inverse of the cumulative distribution function $F_E$. 
 As the distribution is continuous, there is an infinity of different energy quantity $E_\th$  i.e. an infinity of \emph{players types} in the obtained game, where a \emph{type} refers to the definition of a set of feasible actions and a payoff function.
 The operator has two questions: how to characterize an equilibrium of this nonatomic game with an infinity of players types and how to compute such an equilibrium. This paper provides answers to those two questions.

\medskip
The game described above belongs to the class of aggregative games. In such a game, a player's payoff is determined by her own action and the aggregate of all the players' actions \cite{Corchon1994}. The setting of aggregative games is particularly relevant to the study of nonatomic games \cite{schmeidler1973equilibrium}, i.e. games with a continuum of players. There, a player has an interaction with the other players only via an aggregate-level profile of their actions%
, while she has no interest or no way to know the behavior of any particular player or the identity of the player making a certain choice.

Nonatomic games are readily adapted to many situations in industrial engineering or public sectors where a huge number of users, such as traffic commuters and electricity consumers, are involved. 
These users have no direct interaction except through the aggregate congestion or consumption to which they are contributing collectively.
These situations can often be modeled as a congestion game, a special class of aggregative games, both in nonatomic version and finite-player version.
The latter, called atomic congestion game, was formally formulated by Rosenthal in 1973 \cite{rosenthal1973network}, while related research work in transportation and traffic analysis, mostly in the nonatomic version, appeared much earlier \cite{wardrop1952some,Bec56}.
 The theory of congestion games has also found numerous applications in telecommunications \cite{orda1993competitive}, distributed computing \cite{altman2002nash}, energy management \cite{atzeni2013demand}, and so on. 
 
Nonatomic games are mathematical tools adapted to the modeling  of interactions between a very large number of agents. Practical cases exist where a nonatomic model is intuitive and straightforward as when the modeler has an easier access to a description of the population through a parametric distribution of the types, as illustrated in the example. 

As many distributions used in practice (e.g. normal distribution) are continuous, this implies that the nonatomic game obtained using these distributions will have an infinite number of players' types.

The concept of equilibrium in nonatomic games is captured by the so called Wardrop equilibrium (WE) \cite{wardrop1952some}. A nonatomic player neglects the impact of her deviation on the aggregate profile of the whole population's actions, in contrast to a finite player. 

For the computation of WE, existing results are limited to particular classes of nonatomic games, such as population games \cite{MaynardSmith82,HofSig98,San11}, where only a finite number of types of players are considered, each type sharing the same finite number of pure actions and the same payoff function. 

The objective of this paper is to provide a model of nonatomic aggregative games with \emph{infinitely many compact convex pure-action sets and infinitely many payoff functions}---in general a specific action set and a specific payoff function for each nonatomic player---
then introduce a general form of coupling aggregative constraints into these games, define an appropriate notion of equilibrium, study the properties such as existence and uniqueness of these equilibria and, finally, their computation through an approximation.

\medskip

\paragraph{Main results}  
After defining a pure-action profile in a nonatomic game where players have specific compact convex pure-action sets lying in $\rit^T$, and specific cost functions, convex  in their own action variable, %
 \Cref{thm:agg_wardrop} characterizes a WE as a solution to an infinite-dimensional %
 variational inequality (IDVI). 

Using the IDVI formulation, we extend this equilibrium notion to the case of a game with coupling aggregative constraints, by defining variational Wardrop equilibrium (VWE). 
\Cref{th:exist_ve_inf} proves the existence of WE and VWE in monotone nonatomic games by showing the existence of solutions to the characteristic IDVI. 

In \Cref{th:unique_vwe}, we establish the uniqueness of WE and VWE in case of strictly monotone or aggregatively strictly monotone games. The definition of monotone games is an extension of the stable games \cite{HofSan09}, also called dissipative games \cite{sorin2015finite}, in population games with a finite types of nonatomic players to the case with infinitely many types.

 In the case where the nonatomic aggregative game has only a finite number of types of players, we define the notion of symmetric action profiles and symmetric VWE (SVWE), describing situations where all players of the same type play the same action. \Cref{prop:SVWEfiniteChar} shows that SVWEs are characterized as solutions of a \emph{finite}-dimensional VI. Besides, \Cref{prop:exist_svwe} shows that, under monotonicity assumptions, there always exists an SVWE.

\Cref{thm:converge_with_u} is the main result of this paper. 
It shows that, for a sequence of finite-type approximating games, if the finite number of pure-action sets and cost functions converge to those of  the players of a monotone nonatomic aggregative game, and if the aggregative constraint %
converges to the aggregative constraint of the infinite nonatomic game, then any sequence of SVWE associated to the  sequence of approximating games converges in pure-action profile or in aggregate action profile to the VWE of the infinite-type game.
We provide an upper bound on the distance between the approximating SVWE and the VWE, specified as a function of the parameters of the approximating finite-type games and the initial infinite nonatomic game.

This result allows the construction of a sequence of   finite-type approximating games and associated SVWEs %
 so as to approximate the infinite-dimensional VWE in the special class of strongly or aggregatively strongly monotone nonatomic aggregative games, with or without aggregative constraints. 
 Since the resolution of finite-dimensional variational inequalities---characterizing  SVWEs---is computationally tractable \cite{facchinei2007finite}, it follows from our results that a VWE of a nonatomic game with infinitely many types can be approximated with arbitrary precision.  

\Cref{subsec:construction} shows how to construct an finite-type approximating games sequence for two general classes of nonatomic games. 
\ref{app:nonsmooth} gives the main ideas to extend all our results to the case where players have nonsmooth subdifferentiable cost functions: to get easily to the key ideas, we focus on the smooth case in the body of the paper. \ref{app:nash-wardrop} explains how we can use the same arguments to show the convergence Nash equilibria of atomic finite-player games (instead of nonatomic finite-type games) to a VWE of a nonatomic game.

\medskip

\paragraph{Related work} 
Extensive research has been conducted on WE in nonatomic congestion games via their formulation with variational inequalities \cite{MarPat2007}. %
In addition to their existence and uniqueness, the computational and dynamical aspects of equilibria as solutions to variational inequalities have also been studied  \cite{Smith1984a,ZhuMarcotte1994,ZhangNa1997,CPP2002}. 
However, in most cases, the variational inequalities involved have finite dimensions, as opposed to the case of WE in this paper. 
 Marcotte and Zhu \cite{marcotte1997equilibria} consider nonatomic players with continuous types (leading to a characterization of the WE as an IDVI) and studied the equilibrium in an aggregative game with nonatomic players differentiated through a linear parameter in their cost function.

Convergence of some dynamical systems describing the evolution of pure-action distribution in the population of a non atomic game has been established for some particular equilibria in some particular classes such as linear games \cite{TayJon78}, potential games \cite{Bec56,San00} and stable games \cite{smith84stability,HofSan09}. 
Algorithms corresponding to discretized versions of such dynamical systems for the computation of WE have been studied, in particular for congestion games \cite{Friesz1994daytoday,ZhangNa1997}.

In engineering applications of nonatomic games such as the management of traffic flow or energy consumption, individual commuters or consumers often have specific choice sets due to individual constraints, and specific payoff functions due to personal preferences. 
Also, unlike for a transportation user who usually chooses a single path, an electricity consumer as modeled in the example above faces a resource allocation problem where she has to split the consumption of a certain quantity of energy over different time periods.
 Hence, her pure-action set is no longer a finite, discrete set as a commuter but a compact convex set in $\rit^T$ where $T$ is the total number of time periods. 
 Few results exist for the computation of pure-action  WE in the case where players have continuous action sets and in the case where there are infinitely many different types (i.e. action sets and payoff functions) of players. %
  For example, \cite{schmeidler1973equilibrium} shows the existence of equilibrium in nonatomic games with finite action sets. Mas-Colell \cite{MASCOLELL1984} and Carmona and Podzeck \cite{Carmona2009large} consider compact strategy sets and show the existence of mixed strategy equilibria, and do not consider the case of aggregative games and pure-strategy equilibria. In their model, all players share the same actions set. %
   Besides, most of the existing work assumes smooth cost functions of players which is somewhat restrictive in applications, as for instance electricity tariffs or tolls are usually not continuous.

Similarly, the subject of nonatomic games with (aggregative) coupling constraints has only been partially addressed. Coupling  constraints at an aggregative level are to be  considered in many of the above-mentioned applications, as also mentioned in \cite{gramm2017}:  
 for instance, when modeling the electricity consumption (see above example),  some capacity constraints of the network  or ramping constraints on the variation of total energy consumption between time periods are natural to consider from an engineering point of view.  
As seen in this paper, the presence of coupling constraints is not a simple artifact, as it adds non trivial difficulties in the analysis of WE and their computation. Indeed, an appropriate definition of equilibrium is already not obvious. 
  An analog to the so-called generalized Nash equilibrium \cite{harker1991gne} for finite-player games does not exist for nonatomic games because a nonatomic player's behavior has no impact on the aggregative profile. 
  Moreover, dynamical systems and algorithms used to compute WE in population games cannot be straightforwardly extended to this case. 
  Indeed, in these dynamics and algorithms, players adapt their strategies unilaterally in their respective strategy spaces, which can well lead to a new strategy profile violating the coupling constraint.

Several works have quantified the relationship between Nash and Wardrop equilibria, a subject close to this paper, as shown in \ref{app:nash-wardrop}. 
Haurie and Marcotte \cite{haurie1985relationship} show that in a sequence of atomic splittable games where atomic splittable players are replaced by smaller and smaller equal-size players with constant total weight, Nash equilibria converge to the WE of a nonatomic game. 
Their proof is based on the convergence of variational inequalities corresponding to the sequence of Nash equilibria, a technique similar to the one used in this paper. Wan \cite{wan2012coalition} generalizes this result to composite games where nonatomic players and atomic splittable players coexist, by allowing the atomic players to replace themselves by players with heterogeneous sizes.

Gentile et al. \cite{gentile2017nash} consider a specific class of finite-player aggregative games with linear coupling constraints. They use the variational inequality formulations for the unique generalized Nash equilibrium and the unique generalized Wardrop-type equilibrium (which consists in letting each finite player act as if she was nonatomic) of the same finite-player game to show that, when the number of players grows, the former can be approximated by the latter. Their results are different from ours, as we consider nonatomic games with players of infinitely-many different types instead of finite-player games. Consequently, we consider VWE and symmetric VWE instead of generalized equilibria, which do not exist in nonatomic games. In contrast to generalized equilibria, a variational equilibrium is not characterized by a best reply condition for each of the finite or nonatomic players,  as shown in \Cref{sec:approx_games}. We also consider a general form of coupling constraints, and extend our results to nonsmooth cost functions, as shown in \ref{app:nonsmooth} (we focus on the differentiable case in the body of the paper).
 
Milchtaich \cite{milchtaich2000generic} studies finite and nonatomic crowding games (similar to nonatomic aggregative games), where players have finitely many pure actions, and shows that, if each player in an $n$-person game is replaced by $m$ identical replicas with constant total weight, pure Nash equilibria generically converge to the unique equilibrium of the limit nonatomic game as $m$ goes to infinity. His proof is not based on a variational inequality formulation.
\medskip

\paragraph{Structure} The remaining of the paper is organized as follows. %
\Cref{sec:nonatomic} introduces the definitions of nonatomic aggregative games with and without aggregative constraints. After defining WE and VWE equilibrium in the case of an infinite number of players' types, we show, under monotonicity assumptions, the existence and uniqueness of equilibria there via generalized IDVIs. 
In the case of finite-type games, we define the notions of symmetric profiles and SVWE and show their characterization through finite-dimensional VIs and their existence.
In \Cref{sec:approx_games}, we give the definition of a sequence  of finite-type approximating games  for a nonatomic aggregative game with or without coupling constraints, and present the main theorem of the paper on the convergence of the sequence of S(V)WE, associated to the sequence of finite-type approximating games, to the (V)WE of the nonatomic game. 
In \Cref{subsec:construction}, the construction of   sequences of finite-type approximating games is shown for two important classes of nonatomic aggregative games. 
In \Cref{sec:example_energy}, we step back to the flexible energy example given above, and derive our results to the computation of an SVWE in this framework. 

Last, in \ref{app:nonsmooth}, we show how our results extend to the case of nonsmooth cost functions and in \ref{app:nash-wardrop}, we show how the results can be adapted to prove the convergence of Nash equilibria to a VWE of a nonatomic game.
\medskip
 
\paragraph{Notation} Vectors are denoted by a bold font (e.g. $\xx$) as opposed to scalars (e.g. $x$). 

The transpose of vector $\xx$ is denoted by  
 $\xx^\tau$. 

The closed  ball in a metric space, centered at $\xx$ and of radius $\eta$, is denoted by $B_\eta(\xx)$. 

For a nonempty convex set $C$ in a Hilbert space $\mathcal{H}$ (over $\rit$), 
\begin{itemize}
\item  $T_C(\xx)=\{\yy\in \mathcal{H}: \yy=0 \text{ or } \exists (\xx_k)_k \text{ in } C \text{ s.t. } \xx_k\not\equiv \xx, \xx_k\rightarrow \xx, \frac{\xx_k - \xx}{\|\xx_k-\xx\|}\rightarrow \frac{\yy}{\|\yy\|}\}$ is the \emph{tangent cone} of $C$ at $\xx\in C$; 
\item $\spa C=\{\sum_{i=1}^k \alpha_i \xx_i : k\in \nit, \alpha_i\in\rit, \xx_i\in C\}$ is the \emph{linear span} of $C$;
\item  $\aff C=\{\sum_{i=1}^k \alpha_i \xx_i : k\in \nit, \alpha_i\in\rit, \sum_i \alpha_i = 1, \xx_i\in C\}$ is the  affine hull of $C$;
\item $\rlt C=\{\xx\in C: \exists \eta>0 \text{ s.t. } B_\eta(\xx) \cap \aff C \subset C\}$ is the \emph{relative interior} of $C$;
\item $\rbd C$ is the \emph{relative boundary} of $C$ in $\mathcal{H}$, i.e. the boundary of $C$ in $\spa C$.
\end{itemize}

The \emph{inner product} of two points $\xx$ and $\yy$ in any Euclidean space $\rit^T$ is denoted by $\langle \xx, \yy \rangle= \sum_{i=1}^T x_i y_i$. The $l^2$-\emph{norm} of $\xx$ is denoted by $\norm{\xx} \eqd \langle \xx, \xx \rangle^{1/2}$. 

We denote by $L^2([0,1], \rit^T)$ the Hilbert space of  measurable functions from $[0,1]$  (equipped with the Lebesgue measure $\mu$) to $\rit^T$ that are square integrable with respect to the Lebesgue measure. The \emph{inner product} of two vector functions $F$ and $G$ is denoted by $\langle F, G \rangle_2 = \int_{0}^1 \langle F(\th), G(\th) \rangle \dth$. The Hilbert space $L^2([0,1], \rit^T)$ is endowed with \emph{$L^2$-norm}: $\|F\|_{2}=\langle F, F \rangle_{2}^{1/2}$.

The distance between a point $\xx$ and a set $A$ is denoted by $d_m(\xx,A)\eqd\inf_{\yy \in A}\norm{\xx-\yy}_m$, where $m$ is omitted or is equal to $2$, depending on whether we consider an Euclidean space or  $L^2([0,1], \rit^T)$. 

Similarly, the \emph{Hausdorff distance} between two sets $A$ and $B$ is denoted by $ d_{H,m} (A, B )$, which is  defined as $\max \{\sup_{\xx\in A}d_m(\xx,B), \sup_{\yy \in B}d_m(\yy,A)\}$.

For a function $(\xx,\xxag)\mapsto f(\xx,\xxag)$ of two explicit variables, convex in $\xx$, we denote by $\deriv_1 f(\xx, \xxag)$ the differential of function $f(\cdot,\xxag)$ for any fixed $\xxag$, except in \ref{app:nonsmooth} in which $\partial$ (resp. $\partial_1$) is used to denote the (resp. partial) subdifferential.

\section{Monotonicity, Coupling Constraints and Symmetric Equilibrium}\label{sec:nonatomic}
\subsection{Nonatomic aggregative games}
In nonatomic aggregative games considered here, players have compact pure-action sets, and heterogeneous pure-action sets as well as heterogeneous cost function. This model is in line with  Schmeidler's seminal paper \cite{schmeidler1973equilibrium}.
It differs from most of the population games studied in game theory \cite{HofSig98,San11}, in which nonatomic players are grouped into several populations, with players in the same population having the same finite pure-action set  and the same cost function.

\begin{definition}[Nonatomic aggregative game]\label{def:nonatomicGame} 
A \emph{nonatomic aggregative game} $\Gna$ is defined by:
\begin{enumerate}[i),wide]
\item a continuum of players represented by the points on the real interval $\Theta=[0,1]$ endowed with Lebesgue measure;
 \item  a set of feasible pure actions $\X_\th\subset \rit^T$ for each player $\th\in \Theta$,
 with $T\in \nit^*$ a constant; 
\item a cost function $\X_\th\times \rit^T \rightarrow \rit:  (\xx_\th, \xxag)  \mapsto f_\th(\xx_\th, \xxag)$ for each player $\th$,  where $\xxag=(\xag_t)_{t=1}^T$ and $\xag_\t\eqd \int_0^1 \xx_{\th',t}\dth'$ refers to an aggregate-action profile, given action  profile $(\xx_{\th'})_{{\th'}\in\Theta}$ for the population $\Theta$. 
\end{enumerate}

The set of feasible pure-action profiles is defined by:
\begin{equation*} %
\FX\eqd \left\{ \xx \in L^2([0,1],\rit^{T}) \   : \ \forall  \,\th \in \Theta , \xx_\th \in \X_\th \right\}. 
\end{equation*}

Denote the game by $\Gna=(\Theta, \FX, (f_\th)_{\th\in\Theta})$.
\end{definition}
\begin{remark}
The definition of a nonatomic aggregative game asks the pure-action profile $\xx$ to be a measurable and integrable function on $\Theta$ instead of simply being a collection of $\xx_\th\in \X_\th$ for $\th\in \Theta$. In other words, a coupling constraint is inherent in the definition of nonatomic aggregative games and the notion of WE. This is in contrast with finite-player games.
\end{remark}
The set of feasible aggregate actions is $\Sxag\eqd \{ \xxag \in \rit^T  : \exists\, \xx \in \FX \text{ s.t. } \txt\int_0^1 \xx_\th \dth = \xxag \}$.

Further assumptions are needed for $\FX$ to be nonempty and for the existence of an equilibrium.

\begin{assumption}[Nonatomic pure-action sets]\label{ass_X_nonat}
The correspondence $\X: \Theta\rightrightarrows \rit^T, \th \mapsto \X_\th$ has nonempty, convex, compact values. Moreover, for all $\th\in \Theta$,  $\X_\th \subset B_{\diamX}(\mathbf{0})$, with $\diamX>0$ a constant.
\end{assumption}
Under \Cref{ass_X_nonat}, a sufficient condition for $\xx$ to be in $L^2([0,1],\rit^{T})$ is that $\xx$ is measurable.  
\begin{notation}
We denote by $\M=[0,\diamX+1]^T$ the hypercube in $\rit^T$ of edge $R+1$. 
\end{notation}

\begin{assumption}[Measurability] \label{ass_measur}
 The correspondence $\X: \Theta\rightrightarrows \rit^T, \th \mapsto \X_\th$ has a measurable graph $Gr_\X =\{(\th,\xx_\th) \in \rit^{T+1}: \th \in \Theta, \xx_\th \in \X_\th  \}$, i.e. $Gr_\X$ is a Borel subset of $\rit^{T+1}$.
 
 The function $Gr_\X \rightarrow \rit^T:  (\th,\xx_\th) \mapsto  f_\th(\xx_\th, \yyag)$ is measurable for each $\yyag \in \rit^T$.
\end{assumption}

\begin{assumption}[Nonatomic convex cost functions] \label{ass_ut_nonat}  
For all $\th$, $f_\th$ is defined on $(\M')^2$, where $\M'$ is a neighborhood of  $\M$, and:
\begin{enumerate}[i),wide]
\item for each $\th\in \Theta$, function $f_\th$ is continuous. In particular, $f_\th$ is bounded on $\M^2$;
\item  for each $\th\in \Theta$ and each  aggregate profile $\yyag\in \M$, $\xx \mapsto f_\th(\xx, \yyag)$ is differentiable and convex on $\M'$;
\item there is $\Bdf>0$ such that $\tnorm{\deriv _1 f_\th(\xx_\th,\yyag)} \leq \Bdf$ for each $\xx_\th\in \M$, each $\yyag\in \M$, and each $\th\in \Theta$.
\end{enumerate}

\end{assumption}
\begin{remark}
\Cref{ass_ut_nonat}.iii) implies that $f_\th(\cdot,\cdot)$'s are Lipschitz in the first variable with a uniform Lipschitz constant $\Bdf$ on $\M^2$ for all $\th$. 
\end{remark}

We also need the continuity of the derivative of cost functions in the second (aggregate) variable:
\begin{assumption} \label{assp:gradientContinuity}
For each $\th\in \Theta$ and  each $\xx_\th\in\M$, the function $ \yyag \mapsto \deriv_1 f_\th(\xx_\th, \yyag)$ is continuous on $\M$.
\end{assumption}

Wardrop equilibrium extends the notion of Nash equilibrium in the framework of nonatomic games, where a single player of measure zero has a negligible impact on the others.
\begin{definition}[Wardrop Equilibrium (WE), \cite{wardrop1952some}]
 
A pure-action profile $\sxx\in \FX$ is a pure \emph{Wardrop equilibrium} of nonatomic aggregative game $\Gna$ if we have, with $\sxxag= \int_{\th\in\Theta} \sxx_\th \dth$:
\begin{equation*}
 f_\th (\sxx_\th, \sxxag) \leq f_\th (\xx_\th, \sxxag), \quad\forall \xx_\th \in \X_\th , \ \forall \, a.e.\,  \th \in \Theta\ .
\end{equation*}
\end{definition}
All the actions and equilibria in this paper are pure, hence from now on, we no longer emphasize it.  
\smallskip

Before characterizing WE by infinite-dimensional VI (IDVI), let us introduce some notions and a technical assumption ensuring that the IDVI is well-defined.

\begin{lemma}\label{lm:bestreply} 
 For all $\xx\in L^2([0,1],\M)$, the function  $\g_{\xx}$ defined from $\Theta$ to $\rit^T$ by
\begin{equation}\label{eq:sousgra}
\g_{\xx}(\th)\eqd \deriv_1 f_\th(\xx_\th, \txt\int \xx)\, , \quad \forall \th\in \Theta \, ,\; \forall \xx\in L^2([0,1],\M) 
\end{equation}
is measurable on $\Theta$. Consequently, $\g_{\cdot}$ is a mapping from $L^2([0,1],\M)$ to $L^2([0,1],\rit^T)$.
\end{lemma}
\begin{proof}[Proof of \Cref{lm:bestreply}]
For $n \in \nit$ large enough, $\g_{\xx,n} \eqd \th \mapsto n \big({f_\th(\xx_\th + \frac{1}{n}, \int\xx) -f_\th(\xx_\th,\int\xx) } \big)$ is well defined. It is measurable according to \Cref{ass_measur}. Thus $ \g_{\xx} = \lim_{n} \g_{\xx,n}$ is also measurable as a limit of measurable functions.

 \end{proof}

\begin{theorem}[IDVI formulation of WE]\label{thm:agg_wardrop}%
Under \Cref{ass_X_nonat,ass_ut_nonat,ass_measur}, $\sxx \in \FX$ is a WE of nonatomic aggregative game $\Gna$ if and only if either of the following two equivalent conditions is true:
\begin{subequations}
\begin{align}
\forall \,a.e. \, \th \in \Theta, \quad & \langle \deriv_1 f_\th(\sxx_\th, \sxxag), \xx_\th - \sxx_\th \rangle \geq 0\, , \quad \forall \xx_\th \in \X_\th  \ , \label{cond:ind_opt}\\
 & \int_{\Theta} \langle \g_{\sxx}(\th), \xx_\th - \sxx_\th \rangle \dth \geq 0\, , \quad \forall \xx\in \FX\ . \label{cond:agg_eq}
 \end{align}
\end{subequations}
\end{theorem}

\begin{proof}[Proof of \Cref{thm:agg_wardrop}]
Given $\sxxag$, \eqref{cond:ind_opt} is a necessary and sufficient condition for $\sxx_\th$ to minimize the convex function $f_\th (., \sxxag)$ on $ \X_\th$. Condition \eqref{cond:ind_opt} implies condition \eqref{cond:agg_eq} because of \Cref{lm:bestreply}.

For the converse, suppose that $\sxx\in \FX$ satisfies condition \eqref{cond:agg_eq} but not \eqref{cond:ind_opt}. 
Then there must be a subset $\Theta'$ of $\Theta$ with strictly positive measure such that for each $\th \in \Theta'$, 
$\sxx_\th \notin \Y_\th \eqd \arg\min_{\X_\th}f_\th(\cdot,\sxxag)$.
 In particular, for any $\yy_\th \in \Y_\th$, $\langle \g_{\sxx}(\th), \, \yy_\th - \sxx_\th \rangle <f_\th(\yy_\th,\sxxag) - f_\th(\sxx_\th,\sxxag)< 0$. 
A consequence of \Cref{ass_ut_nonat,ass_measur} is that the function $\Theta \times \M \rightarrow \rit: (\theta, \zz) \mapsto f_\th(\zz,\sxxag)$ is a Carath\'eodory function, that is, (i) $f_\cdot(\zz,\sxxag)$ is measurable on $\Theta$ for each $\zz\in \M$, and (ii) $f_\th(\cdot,\sxxag)$ is continuous on $\M$ for each $\th\in \Theta$.  Thus, according to the measurable maximum theorem \cite[Thm. 18.19]{aliprantis2006infinite} applied to $f_\cdot(\cdot,\sxxag)$, there exists a selection $\yy_\th\in \arg\min_{\X_\th} f_\th(\cdot, \sxxag)$ for $\th\in \Theta'$ such that $\Theta' \rightarrow \rit^T: \th \mapsto \yy_\th$ is a measurable function. 
 By defining $\yy_\th = \sxx_\th$ for $\th\notin\Theta'$, one has $\Theta \rightarrow \rit^T: \th \mapsto \yy_\th$ is measurable and hence belongs to $\FX$. However, $\int_{\Theta} \langle \g_{\sxx}(\th), \yy_\th - \sxx_\th \rangle \dth= \int_{\Theta'} \langle \g_{\sxx}(\th), \yy_\th - \sxx_\th \rangle \dth <0  $, contradicting \eqref{cond:agg_eq}. 
\end{proof}

\begin{remark}\label{rm:tangent_infinitedim}
Condition \eqref{cond:ind_opt} is equivalent to $\langle \g_{\sxx}(\th), \yy_\th \rangle \geq 0$ for all $\yy_\th \in T_{\X_\th}(\sxx_\th)$ for each $\th$. It means that no unilateral deviation is profitable. However, since each nonatomic player has measure zero, when considering a deviation in the action profile, one must let players in a set of strictly positive measure deviate: \eqref{cond:agg_eq} means that the collective deviation of players of any set of strictly positive measure increases their cost.%
\end{remark}

The existence of WE is obtained by an equilibrium existence theorem for nonatomic games.
\begin{theorem}[Existence of a WE,  \cite{rath1992direct}]\label{thm:exist_we}
Under \Cref{ass_X_nonat}, \Cref{ass_measur} and \Cref{ass_ut_nonat}.i), if for all $\th$ and all $\yyag\in \M$, $f_\th(\cdot, \yyag)$ is continuous on $\M$, then the nonatomic aggregative game $\Gna$ admits a WE.
\end{theorem}
\begin{proof} 
The conditions required in Remark 8 in Rath's 1992 paper \cite{rath1992direct} on the existence of WE in aggregate games are satisfied. 
\end{proof}
\begin{remark}
No convexity of $f_\th(\cdot,\yyag)$'s are needed for the existence.
\end{remark}

\subsection{Monotone Nonatomic Aggregative Games}\label{subsec:monotone}
For the  uniqueness of WE and the existence of equilibrium notion to be introduced in the next subsection for the case with coupling constraints, let us consider the following notions of \emph{monotonicity} in  nonatomic aggregative games, also sometimes referred to as \emph{stability}.

\begin{definition}\label{def:mono_non}
With notation $\gxx(\th) = \deriv_1 f_\th(\xx_\th, \txt\int \xx)$, for any $\th\in\Theta$ and any  $\xx, \yy \in L^2([0,1],\M)$, we say that the nonatomic aggregative game $\Gna$ is 
\begin{enumerate}[i),wide]
\item \emph {monotone} if
\begin{equation}\label{cd:mono_non}
\int_{\Theta} \langle \gxx(\th)  - \gyy(\th), \xx_\th - \yy_\th \rangle \dth \geq 0, \quad \forall \xx, \yy \in L^2([0,1],\M)  \ .
\end{equation}

\item \emph{strictly monotone} if the equality in \eqref{cd:mono_non} holds if and only if $\xx=\yy$ almost everywhere.

\item \emph{aggregatively strictly monotone} if the equality in \eqref{cd:mono_non} holds if and only if $\int \xx=\int \yy$.

\item \emph{strongly monotone} with modulus $\alpha$ if
\begin{equation}\label{cd:strong_mono_non}
\int_{\Theta} \langle \gxx(\th)- \gyy(\th), \xx_\th - \yy_\th \rangle \dth \geq \alpha\|\xx-\yy\|^2_2, \; \forall \xx,\yy \in L^2([0,1],\M) \ .
\end{equation}

\item \emph{aggregatively strongly monotone} with modulus $\beta$ if
\begin{equation}\label{cd:strong_agg_mono_non}
\int_{\Theta} \langle \gxx(\th)- \gyy(\th), \xx_\th - \yy_\th \rangle \dth \geq \beta\|\txt\int \xx-\txt\int \yy\|^2, \;\forall \xx,\yy \in L^2([0,1],\M)  \ .
\end{equation}
\end{enumerate}
\end{definition}
\begin{remark}
A recent paper of Hadikhanloo \cite{Hadikhanloo2017} generalizes the notion of stable games in population games \cite{HofSan09} to monotone games in anonymous games, an extension of population games with players having heterogeneous compact action sets but the same payoff function. 
 He defines the notion of monotonicity directly on the distribution of actions among the players instead of action profile as we do. The two approaches are compatible.
\end{remark}

\paragraph{Example of public products games.} An interesting example of aggregative games are given by cost functions of the form:
\begin{equation}\label{eq:common_form_cost}
f_\th(\xx_\th,\xxag)= \langle \xx_\th , \cc(\xxag) \rangle - u_\th(\xx_\th) \ ,
\end{equation}
where $\cc(\xxag)$ specifies the per-unit cost (or negative of per-unit utility) of each of the $T$ ``public products'', which is a function of  the aggregative contribution $\xxag$ to each of the ``public products''. Player $\th$'s cost (resp. negative of utility) associated to these products is scaled by her own contribution $\xx_\th$.  The function $u_\th(\xx_\th)$ measures the private utility of player $\th$ (resp. negative of private cost) for the contribution $\xx_\th$. 

For instance, in a public goods game, $-c_t (\xag_t)$ is the common per-unit payoff for using public good $t$, determined by the total contribution $\xag_t$, while $-u_\th(\xx_\th)$ is player $\th$'s private cost of supplying $\xx_\th$ to the public goods;  in a Cournot competition, $-c_t (\xag_t)$ is the common market price for product $t$, determined by its total supply $\xag_t$, while $-u_\th(\xx_\th)$ is player $\th$'s private cost of producing $x\thti$ unit of product $t$ for each product $t$; in a congestion game, $c_t(\xag_t)$ is the common per-unit cost for using arc $t$ in a network, determined by the aggregate load $\xag_t$ on arc $t$, while $u_\th(\xx_\th)$ is player $\th$'s private utility of her routing or energy consuming choice $\xx_\th$.

\begin{proposition}\label{prop:monotonemap}
Under \Cref{ass_X_nonat,ass_measur,ass_ut_nonat}, in a \emph{public products} game $\Gna$ (i.e. with cost functions of form \eqref{eq:common_form_cost}), assume that $\cc$ is monotone on $\M$ and, for each $\th$, $u_\th$ is a concave function on $\M$. Then:
\begin{enumerate}[i),wide,labelindent=4pt]
\item  $\Gna$  is a monotone game.
\item If $u_\th$ is strictly concave on $\M$ for all $\th \in \Theta$, then $\Gna$ is a strictly monotone game.
\item If $\cc$ is strictly monotone on $\M$, then $\Gna$ is an aggregatively strictly monotone game.
\item If $u_\th$ is strongly concave on $\M$ with modulus $\alpha_\th$ for each $\th \in \Theta$ and $\inf_{\th\in\Theta}\alpha_\th= \alpha>0$, then $\Gna$ is a strongly monotone game with modulus $\alpha$.
\item If  $\cc$ is strongly monotone on $\M$ with $\beta$, then $\Gna$ is an aggregatively strongly monotone game with modulus $\beta$.
\end{enumerate}
\end{proposition}
\begin{proof}
i) Let $\xx,\yy\in \FX$ and $\xxag=\int \xx$, $\yyag=\int \yy$. For each $\th$, $\deriv_1 f_\th(\xx_\th, \xxag)=\cc(\xxag)-  \nabla u_\th(\xx_\th)  $.  
Then, given $\xx,\yy\in \FX$, we have %
$\langle  \nabla u_\th(\xx_\th)  - \nabla u_\th(\yy_\th) , \xx_\th - \yy_\th \rangle\leq 0$ because $u_\th$ is concave. 
Then 
$\int_0^1 \langle \deriv_1 f_\th(\xx_\th, \xxag)-\deriv_1 f_\th(\yy_\th, \yyag), \xx_\th-\yy_\th\rangle\dth= \langle \cc(\xxag)-\cc(\yyag), \xxag-\yyag\rangle - \int_0^1\langle  \nabla u_\th(\xx_\th)  - \nabla u_\th(\yy_\th), \xx_\th - \yy_\th \rangle \dth\geq 0$ because $\cc$ is monotone. Hence $\Gna$ is a monotone game. 

The proof for ii)-v) is omitted.
\end{proof}

In particular, if $\cc(\xxag)=(c_t(\xag_t))_{t\in T}$, then $\cc$ is monotone if $c_t$'s are all non-decreasing, and $\cc$ is strongly monotone if $c_t$'s are all strictly increasing.

\subsection{Aggregate constraints and VWE}\label{subsec:GWE}
Let us consider the aggregative constraint in nonatomic aggregative game $\Gna$: $\xxag\in A$,  where $A$ is a convex compact subset of $\rit^T$ such that $A\cap \Sxag \neq \emptyset$. Let $\FX(A)$ be a subset of $\FX$ defined by $\FX(A)\eqd \{\xx \in \FX: \xxag=\int \xx \in A\}$. Let us denote the nonatomic aggregative game with aggregative constraint $\xxag \in A$ by $\GnaA$.

A notion of generalized WE similar to the one of generalized Nash equilibrium in finitely-many-player games---where each player does the best she can while not violating the coupling constraints given the choices of the others \cite{harker1991gne}---is not well-defined in a nonatomic game. Indeed, since the impact of a nonatomic player's choice on the aggregative profile is negligible, the feasible action set of a nonatomic player $\th$ facing the choices of the others $\xx_{-\th}$ in a game with coupling constraint is not a well-established notion: either $\int \xx_{-\th} \in A$ then $\X_{\th}=\X$, or $\int \xx_{-\th} \notin A$ then $\X_{\th}=\emptyset$. 
Departing from an action profile in $\FX(A)$, simultaneous unilateral deviations by the players can lead to any profile in $\FX$. If only profiles in $\X(A)$ are allowed to be attained, then one lands on a notion similar to the so-called variational Nash equilibrium in finite-many-player games  \cite{harker1991gne}. 
 Indeed, the most natural notion of equilibrium with the presence of aggregative constraint is the notion of variational Wardrop equilibrium, where feasible deviations are defined on a collective basis.

\begin{definition}[Variational Wardrop Equilibrium (VWE)]\label{def:ve-infinite}
A solution to the following IDVI problem: 
 \begin{equation}\label{cond:ind_opt_ve_inf}
\text{Find } \sxx\in \FX(A) \text{  s.t.  } \int_{\Theta} \langle \g_{\sxx} (\th), \xx_\th - \sxx_\th \rangle \dth \geq 0,\quad \forall \xx \in \FX(A),
 \end{equation}
 is called a \emph{variational Wardrop equilibrium} of  $\GnaA$.%
\end{definition}

\begin{remark}[VWE in the literature]
In the literature of congestion games, the equilibrium notion characterized by VI of form \eqref{cond:ind_opt_ve_inf} but in finite dimension and with smooth cost functions has long been studied. For example, see \cite{larsson1999side,marcotte2004capacitated,correa2004capacitated,zhongal2011} and references therein.
\end{remark}

The following facts are needed for later use.
\begin{lemma} Under \Cref{ass_X_nonat,ass_measur}:\begin{enumerate}[i), wide]
\item  $\FX$ is a nonempty, convex, closed and bounded subset of  $L^2([0,1], \rit^T)$; 
\item $\FX(A)$ is a nonempty, convex and closed subset of $\FX$; \item $\Sxag$ and $A\cap \Sxag$ are nonempty, convex and compact subsets of $\rit^T$.
 \end{enumerate}
 \end{lemma} We omit the proof and only point out that $\FX$ and $\Sxag$ are nonempty because of \Cref{ass_X_nonat} and the measurable selection theorem of Aumann \cite{aumann1969measure}, while aggregate-action set $\Sxag$ is compact by \cite[Theorem 4]{aumann1965integral}.

 \Cref{th:exist_ve_inf} shows the existence of VWE via the VI approach. %

\begin{theorem}[Existence of VWE]\label{th:exist_ve_inf}
Under \Cref{ass_X_nonat,ass_measur,ass_ut_nonat,assp:gradientContinuity}, if a nonatomic aggregative game with coupling constraint $\Gna(A)$ is monotone on $\FX(A)$, then a VWE  exists. 
\end{theorem}
\begin{proof}
We can apply \cite[Corollary 2.1]{ding1996monotoneGVI} which shows that \Cref{cond:ind_opt_ve_inf} has a solution, as:
\begin{itemize}[wide]
\item $\FX(A)$ is bounded, closed and convex in $L^2([0,1],\rit^T)$;
\item$\g_.:L^2([0,1],\M) \rightrightarrows L^2([0,1],\rit^T)$ is a monotone correspondence which is upper hemicontinuous from the line segments in $\FX(A)$ to the weak* topology of $L^2([0,1],\rit^T)$.  Notice that $\g_.$ has closed values. Let us do the proof in the general nonsmooth case.

 Take $\xx$ and $\yy$ in $\FX(A)$, consider sequence $(\xx\supk)_k$ with $\xx\supk \in [\xx,\yy] $ with $\xx\supk \rightarrow \xx$, and sequence $(\g\supk)_k$ such that $\g\supk\in \g_{\xx\supk}$ and $\g\supk \stackrel{\ast}{\rightharpoonup} \g^\infty$ with $\g^\infty\in L^2([0,1],\rit^T)$. Let us show that $\g^\infty\in \gxx$. 
We have $\xxag=\int \xx$ converging to $\xxag\supk=\int\xx\supk$ in $l^2$-norm.

By definition of $\g_.$ and convexity, for each $\zz\in \M$, for each $\th$, $f_\th (\zz_\th, \xxag\supk_\th) \geq f_\th (\xx\supk_\th, \xxag\supk) + \langle \g\supk_\th, \zz_\th-\xx\supk_\th \rangle$. Since $f_\th$ is continuous in both variables, $f_\th (\zz_\th, \xxag\supk)\rightarrow f_\th (\zz_\th, \xxag)$ and $f_\th (\xx_\th\supk, \xxag\supk)\rightarrow f_\th (\xx_\th, \xxag)$. Besides, $\langle \g\supk_\th, \zz_\th-\xx\supk_\th \rangle = \langle \g\supk_\th, \zz_\th-\xx_\th \rangle + \langle \g\supk_\th, \xx_\th-\xx\supk_\th \rangle$, and $\langle \g\supk_\th, \zz_\th-\xx_\th \rangle \rightarrow \langle \g^\infty_\th, \zz_\th-\xx_\th \rangle$ because  $\g\supk \stackrel{\ast}{\rightharpoonup} \g^\infty$, while $ \langle \g\supk_\th, \xx_\th-\xx\supk_\th \rangle\rightarrow 0$ because $\g\supk_\th$'s are uniformly bounded by $\Bdf$. Therefore, $f_\th (\zz_\th, \xxag) \geq f_\th (\xx_\th\supk, \xxag) + \langle \g_\th, \zz_\th-\xx_\th \rangle$ so that $\g^\infty_\th\in \partial_1 f_\th(\xx_\th, \xxag)$. Since the limit of measurable functions is measurable,  $\g$ is measurable. Hence $\g^\infty\in \gxx$  (and $\g^\infty=  \gxx$ in the smooth case), which concludes.
\end{itemize}  
\end{proof}

\begin{theorem}[Uniqueness of VWE]\label{th:unique_vwe}%
Under \Cref{ass_X_nonat,ass_measur,ass_ut_nonat}: 
\begin{enumerate}[i),wide]
\item if  $\Gna(A)$ is strictly monotone on $\FX(A)$, then it has at most one VWE; 
\item if $\Gna(A)$ is aggregatively strictly monotone on $\FX(A)$, then all VWE of $\Gna(A)$ have the same aggregative profile;
  \item if $\Gna$ (without aggregative constraint) is only aggregatively strictly monotone but, for each $\th\in \Theta$ and all $\yyag\in \M$, $f_\th(\xx,\yyag)$ is strictly convex in $\xx$, then there is at most one WE.
\end{enumerate}
\end{theorem}
\begin{proof}
Suppose that $\xx,\yy\in \FX(A)$ are both VWE. Let $\xxag=\int\xx$ and $\yyag=\int \yy$. According to Theorem \ref{thm:agg_wardrop}, we have $ \int_{\Theta} \langle \gxx(\th), \yy_\th - \xx_\th \rangle \dth \geq 0$ and $\int_{\Theta} \langle\gyy(\th), \xx_\th - \yy_\th \rangle \dth \geq 0$. Adding up these two inequalities yields $\int_{\Theta}\langle \gxx(\th) -\gyy(\th), \yy_\th - \xx_\th \rangle \dth \geq 0$. 

i) If $\Gna(A)$ is a strictly monotone game, then $\int_{\Theta}\langle  \gxx(\th) -\gyy(\th), \xx_\th - \yy_\th \rangle \dth = 0$ and thus $\xx=\yy$ almost everywhere.  

ii)-iii) If $\Gna(A)$ is an aggregatively strictly monotone game, then $\int_{\Theta}\langle  \gxx(\th) -\gyy(\th), \xx_\th - \yy_\th \rangle \dth = 0$ and thus $\xxag=\yyag$.  

If there is no aggregative constraint and $f_\th(\cdot,\zzag)$ is strictly convex for all $\zzag\in \M$, then for all $\th$, $\xx_\th$ (resp. $\yy_\th$) is the unique minimizer of $f_\th(\cdot, \xxag)$ (resp. $f_\th(\cdot, \yyag)$). Since $\xxag=\yyag$, one has $\xx_\th=\yy_\th$.
\end{proof}

\subsection{Symmetric VWE with a finite number of types}\label{subsec:SVWE}

A particular class of nonatomic aggregative games is those with only a finite number of types of players, that is, when the sets $\{ \X_\th\}_\th$ and $\{f_\th\}_\th$ are both finite. Consider a nonatomic aggregative game with a set of $I$ types $\I=\{1, \ldots, I\}$.  The player set $\Theta$ is divided into $I$ measurable subsets $\Theta_1, \ldots, \Theta_I$ such that each nonatomic player belonging to $\Theta_i$ is of type $i$. 
Let us denote the common action set of  players in $\Theta_i$ by $\X_i$ and their common cost functions by $f_i$.

Let us consider a particular class of action profiles in these finite-type nonatomic aggregative games, called \emph{symmetric action profiles}:

\begin{definition}[Symmetric action profile and symmetric variational Wardrop equilibrium (SVWE)]
The set of symmetric action profiles, denoted by $\FXS$, is the set of action profiles where players of the same type play the same action:
\begin{equation*}
\FXS \eqd \{\x \in \FX: \xx_\th =\xx_\xi, \forall \th, \xi \in \Theta_i, \forall i\in \I\}
\end{equation*} 
The set of symmetric action profiles satisfying the aggregative constraint is denoted by  
\begin{equation} \label{def:FXSA}
\FXS(A)\eqd \FXS \cap \FX(A)\, .
\end{equation}

A symmetric variational Wardrop equilibrium is a VWE that is symmetric. 
\end{definition}

For any symmetric action profile $\xx \in \FXS$, let the common action of players of type $i\in \I$ be denoted by $\xx_i$, so that the action profile can be specified by $(\xx_i)_{i\in \I}$. Obviously, for $\xx\in \FXS$, for each type $i\in \I$, $\gxx (\th) = \gxx (\xi)= \deriv_1 f_i(\xx_i, \txt\int \xx)$ for all $\th, \xi \in \I$. Let  us abusively denote this quantity by $\gxx(i)$.

\begin{proposition}\label{prop:SVWEfiniteChar}
In a finite-type nonatomic aggregative game $\Gna(A)$ with an aggregative constraint, a VWE is a symmetric one if and only if it is  a solution to the following VI:
\begin{equation}\label{eq:def-pseudo}
 \text{Find } \hxx\in \FXS(A) \text{ s.t. }  \txt\sum_{i\in \I}\langle \g_{\hxx}(i), \mu_i \xx_i - \mu_i\hxx_i \rangle\geq 0,\; \forall \xx\in \FXS(A)\ ,
 \end{equation}
where $\mu_i$ is the Lebesgue measure of $\Theta_i$.
\end{proposition}
\begin{proof}
Since $\FXS(A)\subset \FX(A)$, it is clear that a SVWE, characterized as a solution to the IDVI \eqref{cond:ind_opt_ve_inf}, is a solution to \eqref{eq:def-pseudo}.

Conversely, suppose that $\hxx$ is a solution to the VI problem \eqref{eq:def-pseudo}, let us show that it also solves the IDVI \eqref{cond:ind_opt_ve_inf}. Indeed, for all $\xx \in \FX(A)$,
 \begin{align*}
 \int_{\Theta} \langle \g_{\hxx}(\theta), \xx_\th - \hxx_\th \rangle \dth 
& =\txt \sum_{i\in \I}    \int_{\Theta_i} \langle \g_{\hxx}(\theta), \xx_\th - \hxx_\th \rangle \dth = \sum_{i\in \I} \langle \g_{\hxx}(i), \int_{\Theta_i} \xx_\th - \mu_i \hxx_i \rangle \\
& = \txt  \sum_{i\in \I} \langle \g_{\hxx}(i), \mu_i \tfrac{\int_{\Theta_i} \xx_\th}{\mu_i} - \mu_i  \hxx_i \rangle \geq 0
 \end{align*}
as for all $\th\in \Theta_i$, $\xx_{\th} \in \X_i$ which is convex, hence $\tfrac{\int_{\Theta_i} \xx_\th}{\mu_i}\in \X_i$, so that \eqref{eq:def-pseudo} can be applied.
\end{proof}

\begin{proposition}[Existence of SVWE]\label{prop:exist_svwe}
  Under \Cref{ass_X_nonat,ass_ut_nonat,assp:gradientContinuity}, a finite-type nonatomic aggregative game $\Gna(A)$  admits a SVWE. 
\end{proposition}
\begin{proof}
First note that the VI problem \eqref{eq:def-pseudo} is equivalent to a finite dimension VI:
\begin{equation}\label{eq:def-pseudo_fd}
\text{Find } \hxx\in \FXS(A)' \text{ s.t. }  \txt\sum_{i\in \I}\langle \deriv_1 f_i(\hxx_i, \txt\sum_{j\in \I} \mu_j \hxx_j) , \mu_i \xx_i - \mu_i\hxx_i \rangle\geq 0,\; \forall \xx\in \FXS(A)'\ ,
\end{equation}
where $\FXS(A)'$ is the finite-dimensional set \begin{equation*}
\FXS(A)' \eqd \{\xx\in \rit^{IT}: \xx\in \txt\prod_{i\in \I} \X_i, \ \sum_{i\in \I} \mu_i \xx_i \in A\}\ .
\end{equation*}

As %
  the mapping $ (\xx_i)_{i\in\I} \mapsto \big(\deriv_1 f_\i(\xx_\i, \sum_j \mu_j \hxx_j) \big)_{i\in\I}$ is continuous from \Cref{assp:gradientContinuity}, 
then \cite[Lemma 3.1]{hartman1966} implies that the VI \eqref{eq:def-pseudo_fd} has a solution on the finite dimensional convex compact $\FXS(A)'$.
\end{proof}

In this paper, only SVWE are considered for finite-type nonatomic aggregative games. For such equilibria, the distribution of different types on $\Theta$ is not relevant, since the equilibrium behavior of each player is only determined by the finite dimensional VI  \eqref{eq:def-pseudo_fd}. Therefore, we shall specify a finite-type nonatomic aggregative game only by the tuple $\big((\mu_i)_{i\in \I},(\X_i)_{i\in \I},(f_i)_{i\in \I},A \big)$. In particular, a symmetric action profile in such a game shall be specified by a vector $(\x_i)_{i\in \I} \in \rit^{IT}$, and the set of symmetric action profiles is nothing else but $\FXS(A)'$.

\section{Approximating an Infinite-type nonatomic aggregative game}
\label{sec:approx_games}
\subsection{Finite-type approximating game sequence}

After introducing (V)WE in nonatomic aggregative games and SVWE in finite-type nonatomic aggregative games with coupling constraints, we  study the relation between these notions.  
As %
 WE is a particular case of VWE when the  aggregate constraint set  is  any subset of $\rit^T$ containing $\Sxag$, we can only consider the case of %
  VWE and SVWE.

This section shows the following result: considering a sequence of equilibria of ``approximating'' finite-type nonatomic aggregative games $(\Gna\esnu(A\esnu))_\snu$ of a nonatomic aggregative game $\Gna(A)$,  where each type of players in $ \Gna\esnu(A\esnu) $ corresponds to a collection of nonatomic players who are similar in their types,
 a sequence of SVWE in $(\Gna\esnu(A\esnu))_\snu$ %
 converges to the VWE of $\GnaA$ when this one is (aggregatively) strongly monotone.

The particularity of SVWE is that it can be characterized by a finite dimensional VI. 
As opposed to the case of infinite dimensional ones, there is a large literature on algorithms for computing solutions of finite dimensional VI (e.g. \cite{facchinei2007finite} and references therein). %
Therefore, the result stated above can be practically used to compute a VWE, solution of an IDVI, with arbitrary precision.

In this section, we always suppose that  \Cref{ass_X_nonat,ass_measur,ass_ut_nonat,assp:gradientContinuity} hold. %
\medskip

Let us consider the following definition of an approximating game sequence:

\begin{definition} \label{def:approx_seq} \textbf{Finite-type Approximating Games   Sequence } \\
A sequence of finite-type nonatomic aggregative games $\{\Gna\esnu(A\esnu)=
\big((\mu\esnu_i)_{i\in \I\esnu},(\X\esnu_i)_{i\in \I\esnu},(f\esnu_i)_{i\in \I\esnu},A\esnu \big) 
: \snu\in \nit^*\}$ with aggregative constraints is a \emph{finite-type approximating game sequence} for the nonatomic aggregative game $\Gna(A)=\big(\Theta,\FX,(f_\th)_\th, A\big)$ with an aggregative constraint if, for each $\snu \in \nit^*$, there exists a partition $(\Theta_0\esnu,\Theta_1\esnu, \dots, \Theta_{I\esnu}\esnu)$ of the set $\Theta$, with  $\I\esnu\eqd \{1, \dots, I\esnu\}$, such that the Lebesgue measure of $\Theta_0\esnu$ is $ \mu_0\esnu= 0$, and if, for each $\i\in \I\esnu$, the Lebesgue measure of $\Theta_\i\esnu$ is $\mu\esnu_i$ while the collection of nonatomic players in $\Theta_\i\esnu$ are getting close to the nonatomic players of type $i\in \I\esnu$ in the sense that, as $\nu \rightarrow +\infty$:
\begin{enumerate}[i),leftmargin=*,wide,labelindent=5pt]
\item  \label{def:approx_seq:subgradientAtomicNonatomic} $\mdset\esnu\eqd\max_{\i \in\I\esnu} \dset_\i\esnu \longrightarrow 0$, where $\dset_\i\esnu$ is the Hausdorff distance  between the action sets of nonatomic players in $\Theta\esnu_\i$ in $\Gna(A)$ and the action set of nonatomic players of type $i\in \I\esnu$ in $\Gna\esnu(A\esnu)$:
\begin{equation} \label{eq:def_dset}
\dset_\i\esnu \eqd \sup_{\th \in \Theta\esnu_\i} d_{H}\left( \X_\th, \X_{\i}\esnu \right)\ ,
\end{equation}
and  $\spa \X\esnu_i = \spa \X_\th, \ \forall \th\in \Theta\esnu_i$. 
\item \label{def:approx_seq:distSets}$\mduti\esnu\eqd\max_{\i\in\I\esnu}\duti_\i\esnu \longrightarrow 0  $, where $\duti\esnu_\i$ measures the %
distance 
between the differential of nonatomic players' cost functions in $\Gna(A)$ and that of their corresponding players' cost functions in $\Gna\esnu(A\esnu)$:
\begin{equation}\label{eq:def_dut}
\duti\esnu_i \eqd  \sup_{\th \in \Theta_\i}\sup_{(\xx,\yyag)\in \M^2}  \norm{   \deriv_1 f\esnu_\i( \xx_i, \yyag) - \deriv_1 f_\th(\xx_{\th}, \yyag) }.
\end{equation}
\item \label{def:approx_seq:distAggSets} $D\esnu \longrightarrow 0$, where $D\esnu\eqd d_H\left(A\esnu,A \right)$  is the Hausdorff distance between the aggregative constraint set $A\esnu\subset \rit^T$ and the aggregative constraint set $A\subset \rit^T$. Besides, $\spa A=\spa A\esnu$ for all $\snu\in \nit^*$.
\end{enumerate} 
\end{definition}
\medskip
Roughly speaking, along a sequence of finite-type approximating games, for each nonatomic player $\theta$ in $\Theta$, the difference between her type and her corresponding type $i$ in the approximating game $\Gna\esnu(A)$ (in the sense that $\theta\in \Theta\esnu_i$) is disappearing as $\snu$ goes to infinity. Also, the aggregate-profile constraint sets of the sequence of approximating games converge to the one in $\Gna(A)$.

Note that, except the last condition on $D\esnu$, the other conditions are independent of the constraint sets $(A\esnu)_\snu$ and $A$.

\begin{remark}
The assumption $ \spa \X\esnu_i = \spa \X_\th, \ \forall \th\in \Theta\esnu_i$ is needed for our proofs because  of the existence of coupling constraints. 
It implies in particular that the nonatomic infinite game considered is such that $\{\spa \X_\th \}_{\th\in\Theta}$ has a finite number of elements. This assumption is natural as, in many models, $\spa \X_\th$ will be the same for all $\th\in\Theta$ (see example of \Cref{sec:example_energy}).

\end{remark}

\begin{remark}
Without loss of generality, we assume $\rlt (A\cap \Sxag)\neq \emptyset$ in this section. Indeed, if the nonempty convex compact set $A\cap \Sxag$ has an empty relative interior, then it is reduced to a point hence the problem becomes trivial. 
\end{remark}

In  \Cref{subsec:construction}, we will construct a sequence of finite-type approximating games  for two fairly general cases of nonatomic aggregative games.

\medskip

In order to compare symmetric action profiles in the approximating games and action profiles in the original game, we introduce the following linear mappings which define an equivalent action profile for a symmetric action profile in an approximating game, and vice versa.

First, define the set $\Symp^{\I\esnu} \eqd \{\xx \in L^2([0,1],\M) :  \xx_\th =\xx_\xi, \forall \th, \xi \in \Theta_i, \forall i\in \I\esnu \}  $
and the mapping $\bpsi\esnu: L^2([0,1],\M) \rightarrow   \Symp^{\I\esnu} $  for each $\snu\in \nit^*$ by 
 \begin{equation}\label{eq:psibar}
 \forall \xx\in L^2([0,1],\M^T), \;
 \bpsi\esnu(\xx)= \big(\,\bpsi\esnu_\th(\xx) \big)_{\th\in\Theta}, \, \text{ where } \forall i\in\I\esnu ,\  \forall \th\in \Theta\esnu_i, \ \bpsi\esnu_\th(\xx)= \tfrac{\int_{\Theta\esnu_\i} \xx_\xi \mr{d} \xi}{\mu\esnu_i} \ .
 \end{equation}
 The interpretation of $\bpsi\esnu$ is that a nonatomic player $\th\in \Theta_i$ (type $i$) adopts the average behavior of players in $\Theta_i\esnu$.%
 
 \medskip 
In the following, we assume that \Cref{ass_X_nonat,ass_ut_nonat,assp:gradientContinuity} also hold for each game $G\esnu$ of a sequence of finite-type approximating games: this appears naturally in many cases if $(G\esnu)_\snu$ is built from $G$, as seen in \Cref{subsec:construction}.  
Finally, let us make the following additional assumption for this section.
\begin{assumption}\label{ass:intpt}
There is a strictly positive constant $\eta$ and an action profile $\bar{\xx}\in \FX$ such that, for almost all $\th\in \Theta$, $d(\bar{\xx}_\th, \rbd \X_\th)>\eta$.
\end{assumption}
It means that the action space of each player has an (aggregatively) nonempty relative interior and that the relative interior is not vanishing along any sequence of players.

\subsection{Convergence of equilibrium profiles and aggregate equilibrium profiles}\label{sec:atomic_with_u}

The following \Cref{thm:converge_with_u} gives the main result of this paper. It shows that a VWE in a strongly monotone nonatomic aggregative game can be approximated by SVWE of a finite-type approximating games sequence, both in the case with and without aggregative constraints. 

Recall that, according to \Cref{th:unique_vwe}, a strongly monotone game is strictly monotone, hence the VWE is unique, while an aggregatively strongly monotone game is aggregatively strictly monotone, hence the aggregate-action profile at VWE is unique. %
\begin{theorem}[Convergence of SVWE to VWE] \label{thm:converge_with_u}
Under \Cref{ass_X_nonat,ass_measur,ass_ut_nonat,assp:gradientContinuity,ass:intpt}, let $(\Gna\esnu(A\esnu))_\snu$ be a sequence of finite-type approximating games for the nonatomic aggregative game $\Gna(A)$ with an aggregative constraint $A$. 
 Let  $\sxx$ be the VWE of $\GnaA$, $\hxx\esnu \in \FX\esnu(A\esnu)$ an SVWE of $\Gna\esnu(A\esnu)$  for each $\snu\in\nit^*$, and $\sxxag$, $\hxxag\esnu$ their respective aggregate-action profiles. 
 Then, there exists a constant  $\rhomin>0$ such that the following results hold with $K_A\eqd\tfrac{\diamX+1}{\rhomin}$:
  \begin{enumerate}[i),wide]
 \item If $\Gna$ is aggregatively strongly monotone with modulus $\beta$, $(\hxxag\esnu)_\snu$ converges to $\sxxag$: for all $\snu\in \nit^*$ such that $\max(\mdset\esnu, D\esnu) < \rhomin $,
\begin{equation} \label{eq:cvg_agg_nou}
\| \hxxag\esnu- \sxxag \|^2 \leq \frac{1}{\beta}  \Big(  (4\Bdf + 1) K_A\max(D\esnu,\mdset\esnu) + (2M+1)\mduti\esnu \Big) .
\end{equation}
\item If $\Gna$ is strongly monotone with modulus $\alpha$, then 
 $(\hxx\esnu)_\snu$ ,  converges to $\sxx$ in $L^2$-norm: for all $\snu\in \nit^*$ such that $\max(\mdset\esnu, D\esnu) < \rhomin $,
 \begin{equation} \label{eq:cvg_indiv}
\| \hxx-\sxx\|^2_2\leq \frac{1}{ \stgccvut} \Big( (4\Bdf +1) K_A\max(D\esnu,\mdset\esnu)  +(2M+1)\mduti\esnu \Big)\ .%
\end{equation}
 \end{enumerate}
If there are no aggregate constraints, one can replace $K_A$ (resp. $D\esnu$) by $\tfrac{1}{2}$ (resp. 0) in \eqref{eq:cvg_agg_nou} and \eqref{eq:cvg_indiv}, and \Cref{ass:intpt} is no longer required.
\end{theorem}

Some notions and a series of lemmas are needed for the proof of \Cref{thm:converge_with_u}.
\medskip

\begin{notation} Let $\Pi\esnu_i(\cdot)$ denote the  (Euclidean)  projection function onto $\X_\i\esnu$ for $\i\in\I\esnu$ and $\Pi_\th(\cdot)$ the projection function onto $\X_\th$ for $\th \in \Theta$. 

Let $\Pi\esnu$ denote the  (Euclidean) projection function onto $\FXS\esnu(A\esnu)\subset \Symp^{I\esnu }$, and $\Pi$ the projection function onto $\FX(A) \subset L^2([0,1], \rit^T; \mu)$.
\end{notation}

Since $\X_\i\esnu$'s, $\X_\th$'s, $\FX(A)$ and $\FXS\esnu(A\esnu)$'s (as defined by \eqref{def:FXSA}) are all convex and closed in their respective Hilbert spaces, the projection functions onto these sets are well defined.

The following \Cref{lem:norm_xat_approx} shows that the players become infinitesimal along a  sequence of finite-type approximating games.
\begin{lemma}\label{lem:norm_xat_approx}
Under \Cref{ass_X_nonat}, for all $\snu \in \nit^*$,  $\|\xx\esnu\|_2 \leq \mdset\esnu+\diamX$  for all $\xx\esnu\in \FXS\esnu$.
\end{lemma}
\begin{proof}
\ifproofs
Let $\xx\esnu_\i \in \X_\i\esnu$ and $\th \in \Theta_\i\esnu$.  By definition of $\dset_\i\esnu$,  $ \big\|{\xx\esnu_i}{} -\Pi_\th\big({\xx\esnu_i}{}\big)\big\| \leq  \dset_\i\esnu$ so that $\norm{\xx\esnu_i} \leq  \big( \dset_\i\esnu  + \big\|\Pi_\th({\xx\esnu_i}{})\big\| \big) \leq  (\dset_\i\esnu  + \diamX) $.
 Then, $\|\xx\esnu\|^2_2=\sum_{i=1}^{I\esnu}\int_{\Theta_i}{\|\xx\esnu_i\|^2 \dth }=\sum_{i=1}^{I\esnu}{\mu\esnu_i}{\|\xx\esnu_i\|^2}\leq\sum_{i=1}^{I\esnu} \mu_i\esnu (\dset_\i\esnu  + \diamX)^2\leq  (\mdset\esnu  + \diamX)^2$. %
\else
Apply the triangle inequality to $\norm{\frac{\xx_i}{\mu_\i\esnu} -P_{\X_\th}\Big(\frac{\xx_i}{\mu_\i\esnu}\Big)}$ where $P_{\X_\th}$ is the projection on $\X_\th$ for a $ \th \in \Theta_i$. 
\fi
\end{proof}

The following lemma shows that the convergence of each type of action set in finite-type game $\Gna\esnu$ to that of her corresponding nonatomic player in $\Gna$, assumed by \Cref{eq:def_dset}, implies the convergence of the product action sets in $L^2([0,1],\M)$.   %
\begin{lemma}[Convergence of $\FXS\esnu$ to $\FX$]\label{lm:FY}
Under \Cref{ass_X_nonat}, for all $\snu\in \nit^*$,
\begin{enumerate}[i),wide]
\item for each $\xx\esnu \in \FXS\esnu$, $d_2(\xx\esnu, \FX)\leq \mdset\esnu$;
\item for each $\xx\in \FX$, $d_2(\bpsi\esnu(\xx), \FXS\esnu)\leq \mdset\esnu$; %
\item for each $i\in \I\esnu$ and each $\xx\esnu_i\in \X\esnu_i$, if $d({\xx\esnu_i}, \rbd {\X\esnu_i})>\dset\esnu_i$, then ${\xx\esnu_i}\in \X_\th$ for all $\th\in \Theta\esnu_i$;
\item for each $i\in \I\esnu$, each $\th\in \Theta\esnu_i$, and each $\xx_\th\in \X_\th$, if $d(\xx_\th, \rbd \X_\th)>\dset\esnu_i$, then $ \xx_\th \in \X_i\esnu$.
\end{enumerate}
\end{lemma}
\begin{proof}
\begin{enumerate}[i),wide]
\item Let $\xx\esnu\in \FXS\esnu$. For each $i\in \I\esnu$ and each $\th \in \Theta_i\esnu$, define $\yy_\th= \Pi_\th ( \xx\esnu ) \in \X_\th$, so that $\|\yy_\th - \xx\esnu  \| \leq \dset_\i\esnu $. Let us show that  $\yy$ is measurable on each $\Theta_i\esnu$, hence measurable on $\Theta$ so that $\yy\in \X$. 
For that,  define $\kappa_i$ on $\Theta\esnu_i\times \M^T$ by $\kappa_i:(\th,\ww)\mapsto \|\xx\esnu_i - \ww\|$. Then, $\kappa_i$ is a Carath\'eodory function. Since the correspondence $\Theta\esnu_i \ni \th \mapsto \X_\th$ is measurable, according to the measurable maximum theorem \cite[Thm. 18.19]{aliprantis2006infinite}, there is a measurable selection of $\xx_\th\in \arg\min_{\X_\th}\kappa_i(\th,\cdot)$. The minimum of $\kappa_i(\th,\cdot)$ on $\X_\th$ is unique and is just $\yy_\th$, hence $\yy$ is measurable on $\Theta_i\esnu$.

Then, $\| \xx\esnu  - \yy\|_2  \leq \mdset\esnu$, which shows that $d_2(\xx\esnu , \FX)\leq \mdset\esnu$. 
\item %
Let $\xx \in \FX$. For each $\i\in\I\esnu$, $\th\in\Theta\esnu_i$, $\|\xx_\th - \Pi\esnu_{\i}(\xx_\th) \|\leq \dset_i\esnu$.
 Since $\X\esnu_{i}$ is a convex subset in $\rit^T$, $\frac{1}{\mu\esnu_i}\int_{\Theta\esnu_i}\Pi\esnu_{i}(\xx_\th)\dth\in \X\esnu_{i}$. 
 Define $\yy\in \FX\esnu$ by $\yy_\th = \yy_i \eqd \frac{1}{\mu\esnu_i} \int_{\Theta\esnu_i}\Pi\esnu_{i}(\xx_\th)\dth \in \X\esnu_i$ for each $\th \in \Theta_i$, for $i\in \I\esnu_i$. 
 Then, from Cauchy-Schwartz inequality:
\begin{align*}
&\|\bpsi\esnu(\xx)-\yy \|^2_2= \sum_{i\in \I\esnu} {\mu\esnu_i}\|\bpsi\esnu_i(\xx)-\frac{1}{\mu\esnu_i}\txt\int_{\Theta\esnu_i}\Pi\esnu_{i}(\xx_\th)\dth\|^2
=\sum_{i\in \I\esnu} \frac{1}{\mu\esnu_i}\|\int_{\Theta\esnu_i}(\xx_\th - \Pi\esnu_{i}(\xx_\th))\dth\|^2
\\ & \leq \sum_{i\in \I\esnu} \frac{1}{\mu\esnu_i}\mu_i\esnu \int_{\Theta_i}\|\xx_\th - \Pi\esnu_{i}(\xx_\th)\|^2 \dth = \sum_{i\in \I\esnu}  \int_{\Theta_i}\|\xx_\th - \Pi\esnu_{i}(\xx_\th)\|^2 \dth \leq \sum_{i\in \I\esnu}\mu_i\esnu (\dset\esnu_i)^2 \leq (\mdset\esnu)^2,
\end{align*}
so that $\|\bpsi\esnu(\xx)-\yy \|_2 
\leq \mdset\esnu$. %
This concludes the proof.
\item Fix $\snu\in\nit^*$, $i\in \I\esnu$ and $\th\in \Theta\esnu_i$. Consider $\xx\esnu_i\in \X\esnu_i$ such that $d(\xx\esnu_i, \rbd {\X\esnu_i})>\eta$ for some $\eta>\dset\esnu_i$. 
Assume that $\xx\esnu_i\notin \X_\th$ i.e. $\big\|\xx\esnu_i-\Pi_{\X_\th}(\xx\esnu_i)\big\|>0$. Let $\yy\esnu_i = \xx\esnu_i +\eta \frac{\xx\esnu_i-\Pi_{\X_\th}(\xx\esnu_i)}{\big\|\xx\esnu_i-\Pi_{\X_\th}(\xx\esnu_i)\big\|}\in {\X\esnu_i}{}$ as $\spa \X_\th \subset \spa \X_i\esnu$.

 Since $\X\esnu_i$ is convex, 
$d(\yy\esnu_i, \X_\th)=\big\|\xx\esnu_i-\Pi_{\X_\th}(\xx\esnu_i)\big\|+\big\|\yy\esnu_i - \xx\esnu_i\big\|>\eta> \dset\esnu_i$, which contradicts the fact that $d(\X_\th, \xx\esnu_i)\leq \dset\esnu_i$. Hence $\xx\esnu_i \in \X_\th$.
\item The proof is similar to that of iii). %
\end{enumerate}
\end{proof}

Note that, because of the convexity assumptions (\Cref{ass_X_nonat}), the sets of aggregate action profiles in the finite-type game $\Gna\esnu(A\esnu)$, obtained by considering symmetric or by considering  non  symmetric profiles are the same, that is:
\begin{equation}
\Sxag\esnu \eqd \Big\{\txt\int_\Theta \xx_\th \dth \ | \xx \in \FX\esnu \Big\} = \Big\{\txt\int_\Theta \xx_\th \dth \ | \xx \in \FXS\esnu \Big\} \ , 
\end{equation}
and the same equality holds for $\Sxag\esnu \cap A\esnu$ when considering aggregate constraint $A\esnu$.

The sequence of sets of aggregate action profiles $(\Sxag\esnu)_\snu$ in games $(\Gna\esnu(A\esnu))_\snu$ \emph{with aggregative constraints}, converges to the set of aggregate-action profiles of the nonatomic aggregative game $\Gna(A)$ \emph{with an aggregative constraint}, as the following lemma says.

\begin{lemma}\label{lem:hausdorff_agg_sets}
Under \Cref{ass_X_nonat},  for $\snu\in \nit^*$, 
\begin{enumerate}[i),wide,labelindent=2pt]
\item $d_H ( \, \Sxag\esnu,\Sxag \,) \leq \mdset\esnu$;
\item for $\xxag\in \rlt\Sxag$, if $d(\xxag, \rbd \Sxag)>\mdset\esnu$, then $\xxag\in \Sxag\esnu$; \\
 for $\xxag\esnu\in \rlt\Sxag\esnu$, if $d(\xxag, \rbd \Sxag\esnu)>\mdset\esnu$, then $\xxag\in \Sxag$;
\item for $\xxag\in \rlt A$, if $d(\xxag, \rbd A)>D\esnu$, then $\xxag\in A\esnu$; \\
 for $\xxag\esnu\in \rlt A\esnu$, if $d(\xxag\esnu, \rbd A\esnu) > D\esnu$, then $\xxag\esnu \in A$; 
\item for $\xxag\in \rlt(\Sxag\cap A)$, if $d(\xxag, \rbd (\Sxag\cap A)) > \max (\mdset\esnu, D\esnu)$, then $\xxag\in \Sxag\esnu\cap A\esnu$; \\
 for $\xxag\esnu\in \rlt(\Sxag\esnu\cap A\esnu)$, if $d(\xxag\esnu, \rbd (\Sxag\esnu \cap A\esnu)) > \max (\mdset\esnu, D\esnu)$, then $\xxag\esnu \in \Sxag\cap A$. 
\end{enumerate}
\end{lemma}
\begin{proof}
\ifproofs
i) 
Fix $\xx \in \FX$. 
Consider $\yy\in \FXS\esnu$ such that $\|\bpsi\esnu(\xx)-\yy \|_2\leq \mdset\esnu$ (cf. \Cref{lm:FY}). 
Then, from Cauchy-Schwarz inequality, $\|\txt\int \bpsi\esnu(\xx) - \txt\int\yy\|^2\leq \|\bpsi\esnu(\xx)-\yy \|^2_2\leq (\mdset\esnu)^2$. 
Hence $d(\int \xx, \Sxag\esnu)\leq \|\txt\int \bpsi\esnu(\xx) - \txt\int\yy\|\leq \mdset\esnu$.

On the other hand, let $\xx\esnu\in \FXS\esnu$, thus $\xxag\esnu \eqd \int \xx\esnu \in \Sxag\esnu$.
 For each $i\in \I\esnu$ and each $\th \in \Theta_i\esnu$, define $\yy_\th= \Pi_\th\big(\xx\esnu_i \big) \in \X_\th$, so that $\|\xx_\th\esnu - \yy_\th \| \leq \dset_\i\esnu $.
  Then, $\|\int \xx\esnu  - \int \yy \| \leq \int \|\xx_\th\esnu- \yy_\th \| \dth   \leq \mdset\esnu $, 
which shows that $d \left( \xxag\esnu, \Sxag\right) \leq \mdset\esnu$ for all $\xxag\esnu\in\ \Sxag\esnu$. 

\smallskip

ii-iii)  The proof is similar to that for \Cref{lm:FY}.iii).

iv) These are corollaries of ii) and iii).
\end{proof}

For the proof of the main theorem, we need to rely on nonatomic profiles that are away from the boundary of the feasible domain, giving us some margin. The existence of such profiles is ensured by the following lemma.

\begin{lemma}\label{lm:intprofile}
Under \Cref{ass_X_nonat,ass:intpt}, there is a strictly positive constant $\rhoz$ and a nonatomic action profile $\zz\in \FX$ such that $\int \zz \in \rlt (\Sxag\cap A)$ and, for almost all $\th\in \Theta$, $d(\zz_\th, \rbd \X_\th)>3\rhoz$.
\end{lemma}
\begin{proof}
Take $\bar{\xx}$ the nonatomic action profile in \Cref{ass:intpt} and an arbitrary $\yy\in\FX(A)$ such that $\int \yy\in \rlt (\Sxag\cap A)$.  

Denote $t=\frac{d(\int\yy, \rbd (\Sxag\cap A))}{3M}$.  Define profile $\zz\in \FX$ by $\zz =\yy - t(\yy -\bar{ \xx}) $. 

Firstly, $\|\int\yy - \int\zz\|=t\|\int\yy-\int\bar{\xx}\|\leq t 2M\leq \frac{2}{3}d(\int\yy, \rbd (\Sxag\cap A))$, hence $\int\zz \in \rlt (\Sxag \cap A)$.

Besides, for any $\th$, $\zz_\th =\yy_\th - t(\yy_\th - \bar{\xx}_\th) $. Since $d(\bar{\xx}_\th, \rbd \X_\th)>\eta$, $\yy_\th\in \X_\th$, and $\X_\th$ is convex, one has $d(\zz_\th, \rbd \X_\th ) > \eta t = \frac{\eta}{3M}d(\int\yy, \rbd (\Sxag\cap A))$. One concludes by defining  $\rhoz \eqd \frac{\eta}{9M}d(\int\yy, \rbd (\Sxag\cap A))$.
\end{proof}

\begin{notation}
Denote $\zzag=\int \zz$ where $\zz$ is the one in \Cref{lm:intprofile} and define $\rhoZ \eqd\frac{1}{3}d(\zzag, \rbd \Sxag\cap A)>0$ and the parameter $\rhomin \eqd \min(\rhoz , \rhoZ)$, appearing in the bounds of \Cref{thm:converge_with_u}. \end{notation}

\begin{figure}[hbtp!]
\centering
\includegraphics[scale=1]{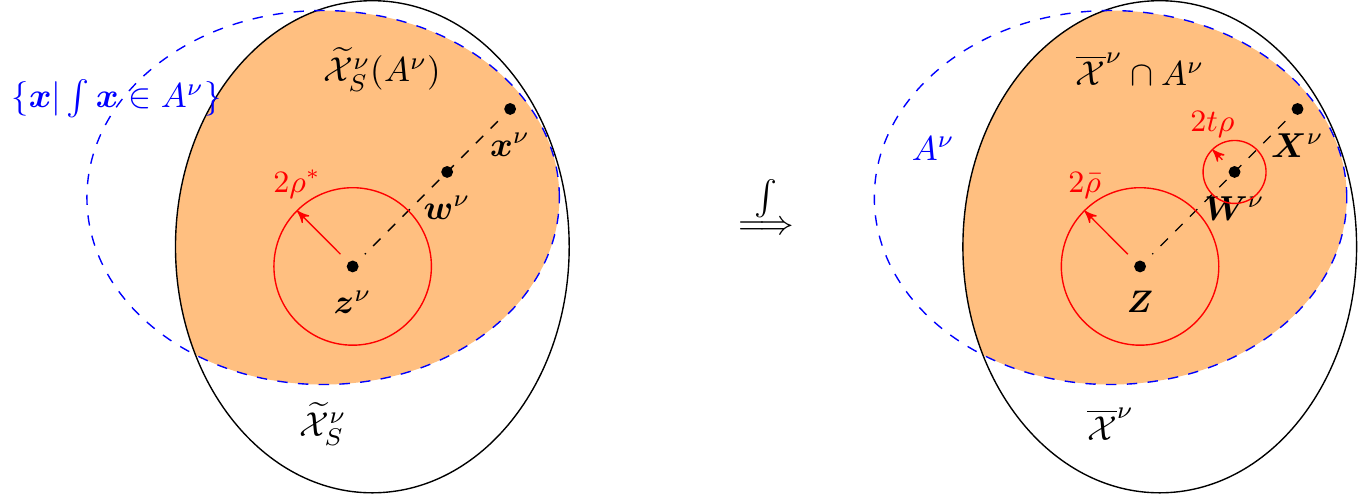}
\caption{Illustration of the mapping $\int_{\Theta} (.)$ between $\FX\esnu(A\esnu)$ and $\Sxag\cap A$ used in \Cref{lem:dist_agg_genized_sets}}
\end{figure}

The following lemma shows that the space of symmetric action profiles in the finite-type game \emph{with aggregative constraint}, $\FXS\esnu(A\esnu)$, is converging to the space of action profiles in the nonatomic aggregative game \emph{with aggregative constraint}, $\FX(A)$.
\begin{lemma}[Convergence of $\FXS\esnu(A\esnu)$ to $\FX(A)$]
\label{lem:dist_agg_genized_sets}
Under \Cref{ass_X_nonat,ass:intpt}, let $K_A=\tfrac{\diamX+1}{\rhomin}$. Then, for all $\snu\in \nit^*$ such that $\max(\mdset\esnu, D\esnu) < \rhomin $,
\begin{enumerate}[i)]
\item for each $\xx\esnu \in \FXS\esnu(A\esnu)$, $d_2(\xx\esnu, \FX(A))\leq 2K_{A}\max(D\esnu,\mdset\esnu)$;

\item  for each $\xx\in \FX(A)$, $d_2(\bpsi\esnu(\xx), \FXS\esnu(A\esnu))\leq 2 K_{A}\max(D\esnu,\mdset\esnu)$. %
\end{enumerate}

\end{lemma}
\begin{proof}
i) Consider $\xx\esnu \in \FXS\esnu(A\esnu)$ and $\xxag\esnu=\int \xx\esnu$. 
 Let $\zz\esnu\eqd \bpsi\esnu(\zz)$ where $\zz$ is defined in \Cref{lm:intprofile}. 
Since for each $\th$,  $d(\zz_\th,\rbd \X_\th) > 3\rho > \mdset\esnu$, one has $\zz\esnu \in \FXS\esnu$ according to \Cref{lm:FY}.iv).

Define $\ww\esnu\eqd\xx\esnu+t(\zz\esnu-\xx\esnu) $ with $t\eqd\frac{\max(D\esnu,\mdset\esnu)}{\rhomin }<1$, then $\ww\esnu \in \FXS\esnu$ from convexity and:
\begin{itemize}[wide]
\item we have $\|\ww\esnu-\xx\esnu\|_2=\max(D\esnu,\mdset\esnu)\frac{\|\zz\esnu-\xx\esnu\|_2}{\rhomin } \leq \max(D\esnu,\mdset\esnu)\frac{2(\diamX+1)}{ \rhomin } \leq 2 K_{A}\max(D\esnu,\mdset\esnu) $;
\item let us show that $\ww\esnu\in \FX$ : from  \Cref{lm:FY}.iii), it is sufficient to show that
 $d(\ww\esnu_i,\rbd \X_i\esnu)  > \mdset_\i\esnu $. For that, we show $d(\zz_i\esnu, \rbd \X_i\esnu) \geq 2\rhomin$ which implies $d(\ww\esnu_i,\rbd \X_i\esnu) \geq 2 t \rhomin > \mdset\esnu$. 
For any arbitrary  $\yy_i \in B_{2\rhomin}( \zz_i\esnu) \cap \spa \X_i\esnu$, let $\yy_\th \eqd \zz_\th + \yy_i-\zz\esnu_i $ for $\th\in\Theta_i\esnu$. Then, $\yy_\th \in  B_{2\rhomin}(\zz_\th) \cap \spa \X_\th$ as $\spa \X_\th = \spa \X_i\esnu$, and $d(\yy_\th, \rbd \X_\th) \geq d(\zz_\th, \rbd \X_\th) -\norm{\zz_\th-\yy_\th} > 3 \rhomin - 2 \rhomin= \rhomin > \mdset\esnu$. Hence, from \Cref{lm:FY}.iv), one has $\yy_\th \in \X_i\esnu$ and, from convexity, $\yy_i= \frac{1}{\mu_i\esnu} \int_{\Theta_i} \yy_\th  \in \X_i\esnu$ which concludes;

\item  let us show  that $\wwag\esnu=\int \ww\esnu\in A$: from   \Cref{lem:hausdorff_agg_sets}.iii), it is sufficient to show that $\wwag\esnu \in \rlt A\esnu $ and $d(\wwag\esnu, \rbd A\esnu) > D\esnu$. First,  since $\max(\mdset\esnu, D\esnu) < \rhoZ$, one has  $d(\zzag, \rbd (\Sxag\esnu\cap A\esnu))   \geq 3 \rhoZ-\max(\mdset\esnu, D\esnu) \geq 2 \rhoZ$. 
The linear mapping $\xx \mapsto \int \xx$ maps the segment linking $\xx\esnu$ and $\zz\esnu$ in $\FX\esnu(A\esnu)$ to a segment linking $\xxag\esnu$ and $\zzag=\zzag\esnu$ in $\Sxag\esnu\cap A\esnu$. 
Hence, by the definition of $\ww\esnu$,
 $B_{2 t \rhomin }(\wwag\esnu) \cap\spa (\Sxag\esnu\cap A\esnu) \subset \Sxag\esnu\cap A\esnu$, because each point in $B_{2 t \rhomin }(\wwag\esnu) \cap\spa (\Sxag\esnu\cap A\esnu)$ is on the segment linking $\xxag\esnu$ and some point in $B_{2 \rhomin }(\zzag) \cap\spa (\Sxag\esnu\cap A\esnu)\subset \Sxag\esnu\cap A\esnu$. We conclude with $d(\wwag\esnu, \rbd(\Sxag\esnu\cap A\esnu))\geq  2 t \rho  =2 \max(D\esnu,\mdset\esnu) > D\esnu $. 

\end{itemize}

ii) The proof is similar and omitted.

\end{proof}

\begin{figure}[ht]\label{fig:proj}
\centering
\includegraphics[scale=1]{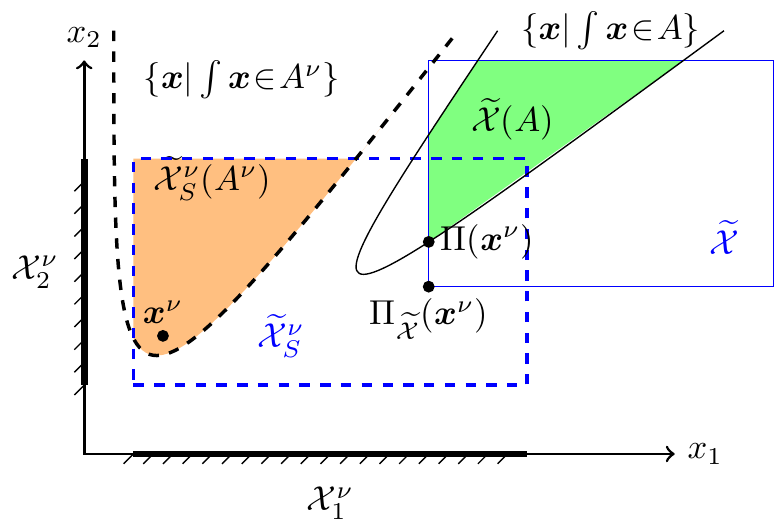}
\caption{Difference between projections on $\FX$ and on $\FX(A)$. (Since it is impossible to draw the graph of a $L^2$ action-profile space with a continuum of players, we illustrate the idea with two players.)}
\end{figure}

\begin{remark}[{Difference between unilateral projections of actions and  collective projection of the action profile}]
\Cref{lem:dist_agg_genized_sets} shows that $d_2(\xx\esnu, \Pi(\xx\esnu))\leq 2K_{A}\max(D\esnu,\mdset\esnu)$ and $d_2(\bpsi\esnu(\xx), \Pi\esnu(\bpsi\esnu(\xx))) \leq 2 K_{A}\max(D\esnu,\mdset\esnu)$, with $\Pi$ and $\Pi\esnu$ the projection functions onto $\FX(A)$ and $\FX\esnu(A\esnu)$. 
\Cref{lem:dist_agg_genized_sets} is of first importance for our proof of \Cref{thm:converge_with_u}: without it, we only have the convergence of individual, i.e. unilateral action spaces in the sequence of approximating games, to the unilateral action spaces in the nonatomic aggregative game, as shown in \Cref{lm:FY}. Without coupling constraints, this should be sufficient in the proof of the convergence of a sequence of SVWE.
 However, in the presence of coupling aggregative constraints, this convergence of unilateral action spaces is not enough.
  Given a profile in $\FXS\esnu(A\esnu)$, unilateral projection of each type $i$ player's action  onto the corresponding nonatomic player's action space in $G(A)$, i.e. from $\X\esnu_i$ to $\X_\th$, cannot guarantee that the resulting action profile is in $\FX(A)$, and vice versa.
   \Cref{lem:dist_agg_genized_sets} shows that, for each action profile $\xx\esnu \in \FXS\esnu(A\esnu)$, its projection on the space of nonatomic action profiles in $\FX(A)$ is  close to $\xx\esnu$, and vice versa. 
\end{remark}

We are finally ready to prove \Cref{thm:converge_with_u}.

\begin{proof}[Proof of \Cref{thm:converge_with_u}]
Fix $\snu\in \nit^*$, define  $\hz\esnu\eqd \Pi(\hxx\esnu) \in L^2([0,1], \rit^T) $.  Then $\hz\esnu\in \FX(A)$ is an action profile in nonatomic aggregative game $\Gna(A)$. 
By the definition of VWE (\Cref{def:ve-infinite}), we have $\int_0^1 \langle \g_{\sxx}(\th) , \ \sxx_\th  - \hz\esnu_\th  \rangle \dth\leq 0$. 

Secondly, by the definition of SVWE,  we have $\int_0^1 \langle \bh_{\hxx\esnu}(\th) , \ \hxx\esnu_\th  - \zz\esnu_\th  \rangle \dth\leq 0  $ for any $\zz\esnu\in\ \FX\esnu(A)$, where $\bh_{\hxx\esnu}(\th)= \nabla_1 f_\th\esnu(\hxx\esnu_\th,\xxag\esnu) = \nabla_1 f_i\esnu(\hxx\esnu_i,\xxag\esnu) \eqd \bh_{\hxx\esnu}(i) $, for each $\th \in \Theta_i\esnu$ and each type $i\in\I\esnu$.

 For all $\i\in \I\esnu$ and $\th\in \Theta\esnu_i$, by the definition of $\duti\esnu_i$ (cf. \Cref{eq:def_dut}), 
we have  $\norm{ \bh_{\hxx\esnu} -   \g_{\hxx\esnu} }_2  \leq \duti\esnu_i$.

Thirdly, $\|\hxx\esnu -\hz\esnu\|_2\leq 2 K_A\max(D\esnu,\mdset\esnu)$ by \Cref{lem:dist_agg_genized_sets}.

With these two results, while noticing that $\hxx\esnu_\th \leq \diamX+\mdset\esnu$ for all $\th$ by \Cref{lem:norm_xat_approx}, one has:
\begin{align}
&  \int_{\Theta} \big\langle   \g_{\sxx}(\th)- \g_{\hxx\esnu}(\th) , \ \sxx_\th-  \hxx\esnu_\th \big\rangle \dth\notag \\
=&  \int_{\Theta} \left\langle \g_{\sxx}(\th) , \ \sxx_\th  - \hz\esnu_\th  \right \rangle \dth + \int_{\Theta} \left\langle \g_{\sxx}(\th) , \  \hz\esnu_\th - \hxx\esnu_\th \right \rangle \dth \notag     \\
& +\int_{\Theta} \big\langle   \g_{\hxx\esnu}(\th) - \bh_{\hxx\esnu}(\th), \  \hxx\esnu_\th - \sxx_\th \big\rangle \dth\notag +\!    \int_{\Theta} \big\langle  \bh_{\hxx\esnu}(\th)  , \ \hxx\esnu_\th \!-\! \sxx_\th  \big\rangle \dth \notag \\
\leq &\, 0\!+  \norm{\g_{\sxx}}_2 \norm{\hz\!-\!\hxx\esnu}_2  \! 
+ \norm{ \g_{\hxx\esnu} - \bh_{\hxx\esnu}}_2 \norm{\hxx\esnu - \sxx }_2 
+ \! J\esnu\notag\\
\leq & \, 2 \Bdf  \, K_A\max(D\esnu,\mdset\esnu)  
+\,(2M+\mdset\esnu)\mduti\esnu  %
+J\esnu   \label{eq:th-firstbound}
\end{align}
where $J\esnu\eqd \int_{\Theta} \big\langle  \bh_{\hxx\esnu}(\th)  , \ \hxx\esnu_\th \!-\! \sxx_\th  \big\rangle \dth =  \sum_{\i\in\I\esnu}  \int_{\Theta\esnu_i} \big\langle \bh_{\hxx\esnu}(i) , \ \hxx\esnu_\th - \sxx_\th  \big\rangle \dth$.

Next, for the VWE $\sxx\in\FX(A)$, let
 $\syy\esnu=\bpsi(\sxx)\in  L^2([0,1],\M)$ and  $\sz\esnu \eqd \Pi\esnu(\syy\esnu) \in \FX\esnu(A\esnu)$:
\begin{align}
 J\esnu & =  \sum_{\i\in\I\esnu}   \big\langle \bh_{\hxx\esnu}(i) , \int_{\Theta\esnu_i}  \hxx\esnu_\th - \sxx_\th  \dth \big\rangle 
 = \sum_{\i\in\I\esnu}  \big\langle \bh_{\hxx\esnu}(i) ,  \mu_i\esnu (  \hxx\esnu_i - \syy\esnu_i )\big\rangle \\
 & = \sum_{\i\in\I\esnu}  \big\langle \bh_{\hxx\esnu}(i) ,  \mu_i\esnu (  \hxx\esnu_i - \sz\esnu_i )\big\rangle  
 + \sum_{\i\in\I\esnu}  \big\langle \bh_{\hxx\esnu}(i) ,  \mu_i\esnu (  \sz\esnu_i - \syy\esnu_i )\big\rangle \\
 & \leq 0 +   (\Bdf+\mduti\esnu) \norm{\sz\esnu- \syy\esnu}_2 \leq  (\Bdf+\mduti\esnu)2K_A\max(D\esnu,\mdset\esnu) \, ,\label{eq:Jnu}
\end{align}
because of the definition of SVWE $\hxx\esnu$ and \Cref{prop:SVWEfiniteChar}, the definition of $\mduti\esnu$ and  \Cref{lem:dist_agg_genized_sets}.$(i)$.

Let us summarize by combining \eqref{eq:th-firstbound} and \eqref{eq:Jnu}, and considering $\snu$ large enough such that $\mduti\esnu, \mdset\esnu \leq 1$:
\begin{equation}\label{eq:unpperboundmono}
\begin{aligned}  
\int_{\Theta}  \big\langle \g_{\sxx}(\th) &-\bh_{\hxx\esnu}(\th) ,  \sxx_\th-  \hxx\esnu_\th \big\rangle \dth \leq \Omega\esnu \text{ \ with \ } \Omega\esnu \eqd (4\Bdf + 1) K_A\max(D\esnu,\mdset\esnu) %
+ (2M+1)\mduti\esnu .
\end{aligned}
\end{equation}
Last,  using the monotonicity definitions (\Cref{def:mono_non}):
\begin{itemize}
\item if $\Gna$ is strongly monotone with modulus $\stgccvut$, then $ \stgccvut  \norm{ \hxx\esnu -  \sxx }^2_2  \leq \Omega\esnu  $;
\item if  $\Gna$ is aggregatively strongly monotone with modulus $\beta$, then $ \beta \| \hxxag\esnu - \sxxag\|^2\leq  \Omega\esnu $, 
\end{itemize}
which lead to the results announced in \Cref{thm:converge_with_u}.
\end{proof}

\begin{remark}
The strong monotonicity of the nonatomic aggregative game $\Gna$, either with respect to action profile or with respect to aggregate-action profile, is essential in this result. In contrast to finite-player games (cf. \cite{ wan2012coalition}), strict monotonicity is not enough to obtain such results using the same techniques. Indeed, since $L^2([0,1],\M^T)$ is only weakly compact, one cannot ensure that $\int_{\Theta}  \langle \g_{\sxx}(\th) -\bh_{\hxx\esnu}(\th) ,  \sxx_\th-  \hxx\esnu_\th \rangle \dth$ tends to $\int_{\Theta}  \langle \g_{\sxx}(\th) -\bh_{\hzz}(\th) ,  \sxx_\th-  \hzz_\th \rangle \dth$ in \eqref{eq:unpperboundmono}, where $\hzz$ is an accumulation point of $(\hxx\esnu)_{\snu}$ in the weak topology.
 \end{remark}

\subsection{Construction of a  sequence of finite-type approximating games}\label{subsec:construction}

As seen in our previous results, a nonatomic player $\th$ is characterized by two elements:  her action set $\X_\th$, and  her gradient $\deriv_1 f_\th$ defined from $\M^2$ to $\rit^T$: $(\xx,\yyag)\mapsto \deriv_1 f_\th(\xx, \yyag)$.

Note that it is the gradient of the cost function $\deriv_1 f_\th$, instead of the cost function $f_\th$ itself, that characterizes a nonatomic player's type. For example, two players $\th$ and $\xi$ with $\X_\th=\X_\xi$ and $f_\th(\xx,\yyag) \equiv f_\xi(\xx,\yyag) + C$ where $C$ is a strictly positive constant can be seen as identical in their behavior.

This section presents the construction of a sequence of finite-type approximating games for a given nonatomic aggregative game $\Gna$ in two particular cases:   1) the player characteristic profile $\th \mapsto (\X_\th, \deriv_1 f_\th)$ is piecewise continuous (cf. \Cref{def:continuity_nonatomic_game}) and, 2) \{$\X_\th,\th\in \Theta\}$ and  \{$f_\th, \th\in \Theta\}$ are respectively polytopes and functions parameterized by a finite number of real parameters.  
Those two cases are fairly general. As illustrated in \Cref{sec:example_energy}, they emerge naturally when the nonatomic game comes from the modeling of a population described by parametric probability distributions, the main motivation for considering infinite nonatomic games.

\paragraph{Case 1: Piecewise Continuous Characteristics -- Uniform Splitting} ~~\label{subsec:approx_uniform}
\begin{definition}[Continuity of nonatomic player characteristic profile]\label{def:continuity_nonatomic_game}
The player characteristic profile $\th \mapsto (\X_\th, \deriv_1 f_\th)$ in nonatomic aggregative game $\Gna$ is \emph{continuous} at $\th \in \Theta$ if, for all $\varepsilon>0$, there exists $\eta>0$ such that: for each $\th'\in \Theta$
\begin{equation}\label{eq:character_continu}
 |\th-\th'| \leq \eta \ \Rightarrow
\begin{cases}
 d_H(\X_\th, \X_{\th'}) \leq \varepsilon \, \\
 \sup_{(\xx,\yyag)\in \M\times \M} \norm{ 
 \deriv_1 f_\th(\xx,\yyag) -
  \deriv_1 f_{\th'}(\xx,\yyag) } \leq \varepsilon  \ .
  \end{cases}
\end{equation}

If \eqref{eq:character_continu} is true for all $\th$ and $\th'$ on an interval $\Theta' \subset \Theta$,  then the player characteristic profile is \emph{uniformly continuous} on $\Theta'$. 
\end{definition}

Assume that the player characteristic profile $\th \mapsto (\X_\th, \deriv_1 f_\th)$ of nonatomic aggregative game $\Gna$ is piecewise continuous, with a finite number $K$ of discontinuity points $\sigma_0=0 \leq \disth_1 <\disth_2 < \dots < \disth_K\leq \sigma_K=1$, and that it is uniformly continuous on $( \sigma_k, \sigma_{k+1})$, for each $k\in  \{0 ,\dots ,K-1\}$.

For $\nu \in \nit^*$, define an ordered set of $I_\snu$ cutting points by $
\{\cutp_i\esnu, i=0,\ldots,I\esnu\} :=\left\{ \tfrac{k}{\nu} \right\}_{0\leq k \leq \nu} \cup \{\sigma_k   \}_{1\leq k \leq K}$
 and the corresponding partition $(\Theta_\i\esnu)_{i\in \I\esnu}$ of $\Theta$ by:
\vspace{-0.15cm}
\begin{equation*}
\Theta_i\esnu= [\cutp_{i-1}\esnu,\cutp_{i}\esnu ) \text{ for }  \i \in \{1, \dots, I\esnu-1\}\  ;\quad \Theta_{ I\esnu}\esnu = [\cutp_{ I_\snu-1}\esnu, 1 ].
\end{equation*}
Hence, $\mu_i\esnu=\cutp_{i}\esnu-\cutp_{i-1}\esnu$. Denote $\bar{\cutp}_{i}\esnu= \frac{\cutp_{i-1}\esnu +  \cutp_{i}\esnu}{2}$.

\begin{proposition}\label{prop:approx_seq_continuous}
Let \Cref{ass_X_nonat,ass_measur,ass_ut_nonat,assp:gradientContinuity} hold, and assume that $\{\spa \X_\th \}_{\th\in\Theta}$ has a finite number of elements. For $\snu \in \nit^*$, consider the finite-type game $\Gna\esnu(A\esnu)$ with aggregative constraint $A\esnu\eqd A$, set of types $\I\esnu\eqd\{ 1  \dots  I\esnu\}$, where for each type $\i\in\I\esnu$:
\vspace{-0.15cm}
\begin{equation*}
\X_i\esnu \eqd \X_{ \bar{\cutp}_{i}\esnu} 
\text{ \ and \ } f_i\esnu(\xx,\yyag)\eqd f_{ \bar{\cutp}_{i}\esnu }\Big(\xx ,\yyag\Big) ,\;  \, \forall (\xx,\yyag)\in \M \times \M.
\end{equation*}
Then $\big(\Gna\esnu(A)\big)_{\snu}=\big( \I\esnu, \FX\esnu,A, (f\esnu_\i)_{\i\in \I\esnu} \big)_{\snu} $ is a sequence of finite-type approximating games of nonatomic aggregative game $\Gna(A)$. 
\end{proposition}
\begin{proof} 
Let us show the three points required by \Cref{def:approx_seq} as follows.

i) Given an arbitrary $\varepsilon>0$, there is a common modulus of uniform continuity $\eta$ such that \eqref{eq:character_continu} holds for all the intervals $(\sigma_k,\sigma_{k+1})$. For $\nu$ large enough, one has, for each $\i \in\I\esnu$, 
 $\mu_\i\esnu< \eta$ so that for all $\th \in \Theta_\i\esnu$, $|\bar{\cutp}_{i}\esnu-\th|< \eta$; hence $d_H\big( \X_\th, \X_\i\esnu\big)=d_H( \X_\th, \X_{ \bar{\cutp}_{i}\esnu}) < \varepsilon$. 

ii) According to the continuity property, for all $(\xx,\yyag)\in \M^2$:
\begin{equation*}
\norm{  \deriv_1 f\esnu_\i\left(\txt {\mu_i\esnu } \xx, \yyag \right) \ -
\  \deriv_1 f_\th(\xx,\yyag) } = \norm{  \deriv_1 f_{ \bar{\cutp}_{i}\esnu}\Big(\xx , \yyag \Big) \ - \ \deriv_1 f_\th(\xx,\yyag)  }< \varepsilon . 
\end{equation*}

To be rigorous, one would  we need to ensure$\spa \X_\th$ to be the same for all $\th\in \Theta\esnu_i$: if not, one can further divide $\Theta\esnu_i$ into a finite number groups so that  players in each group have the same $\spa \X_\th$. This is possible because $\{\spa \X_\th\}_\th$ is finite. %

iii) By definition, $D\esnu=0$.
\end{proof}

\paragraph{Case 2: Finite-dimensions Parameterized Characteristics -- Meshgrid Approximation}~~ %

Assume that the nonatomic aggregative game $\Gna$ satisfy two conditions:

(i) The feasible action sets are $\dimp$-dimensional polytopes: there exists a constant real-valued $\dimp \times T$ matrix $\A$, and a bounded mapping $\bb:\Theta\rightarrow\rit^\dimp$, such that for any $\th$, $\X_\th = \{\xx \in \rit^T  :  \A\xx \leq \bb_\th \}$, which is a nonempty, bounded, closed and convex polytope in $\rit^T$.

(ii) There is a bounded mapping $\ss:\Theta \rightarrow \rit^l$ 
such that for any $\th \in \Theta, f_\th(\cdot,\cdot)= f(\cdot,\cdot\,;\ss_\th)$.   Furthermore, for all $(\xx,\yyag)\in \M^2$, $\deriv_1 f(\xx,\yyag;\cdot)$ is Lipschitz-continuous in $\ss$ and with a Lipschitz constant $L_3$, independent of $\xx$ and $\yyag$.

Denote $\ub_k=  \min_\th b_{\th,k} $, $\ob_k = \max_\th b_{\th,k}$ for $k \in \{1\dots \dimp \}$ and $\us_k= \min_\th s_{\th,k} $, $\os_k = \max_\th s_{\th,k}$ for $k\in \{1\dots l\}$.  The characteristics of player $\th$ are parameterized by point $(\bb_\th, \ss_\th)$ in $\prod_{k=1}^\dimp [\ub_k,\ob_k] \times \prod_{k=1}^l [\us_k,\os_k]$, a compact subset of $\rit^{\dimp+l}$.

Fix $\nu \in \nit^*$, consider a uniform partition of the compact set $\prod_{k=1}^\dimp [\ub_k,\ob_k] \times \prod_{k=1}^l [\us_k,\os_k]$, obtained by dividing each dimension of this compact set into $\snu$ equal parts. Hence, the partition is composed of $I\esnu \eqd \nu^{\dimp+l}$ equal-sized subsets of $\prod_{k=1}^\dimp [\ub_k,\ob_k] \times \prod_{k=1}^l [\us_k,\os_k]$. The  cutting points of the partition are  $\ub_{k,n_k}\eqd \ub_k+\frac{n_k}{\nu}(\ob_k-\ub_k) $ for $k\in \{1,\ldots, \dimp\}$, and $\us_{k,n_k}\eqd \us_k+\frac{n_k}{\nu}(\os_k-\us_k)$ for $k\in \{1,\ldots, l\}$, with $n_k\in \{0,\dots, \nu \}$. Let the set of \emph{vectorial} indices, indexing the partition, be denoted by:
\begin{equation*}
\Gamma\esnu \eqd \{\nn=(n_k)_{k=1}^{\dimp+l} \in \nit^{\dimp+l} \,|\, n_k\in \{1,\ldots, \nu\}\}\ .
\end{equation*}
Define the corresponding partition of the interval $\Theta$: $\Theta=\dot{\bigcup}_{\nn\in \Gamma\esnu} \Theta\esnu_{\nn}$, where:
\begin{align*}
\Theta\esnu_{\nn} \eqd  & \Big\{\th\in \Theta : b_{\th,k} \in [\ub_{k,n_k- 1},\ub_{k,n_k} ) \text{ for } 1\leq k \leq \dimp; \,   s_{\th, k}  \in [\us_{k,n_k-1},\us_{k,n_k}  ) \text{ for } 1\leq k \leq l \Big\}. 
\end{align*}
To be rigorous, when $\ub_{k,n_k}=\ob_{k}$ or $\us_{k,n_k} =\os_k$, the parameter interval is closed at the right. 

Finally, define the set of players $\I\esnu$ as the elements $\nn$ in $\Gamma\esnu$ such that $\mu(\Theta\esnu_{\nn}) >0$.

\begin{proposition}\label{prop:approx_seq_finitedim}
For $\snu \in \nit^*$, let the nonatomic finite-type game $\Gna\esnu(A\esnu)$ with an aggregative constraint $A\esnu\eqd A$, set of types $\I\esnu\eqd\{ \nn\in \Gamma\esnu : \mu(\Theta\esnu_{\nn})>0\}$ and, for each type $\nn\in\I\esnu$, 
\begin{align*}
 \X_{\nn}\esnu &\eqd  \{\xx \in \rit^T  | \A\xx \leq \txt\int_{\Theta\esnu_{\nn}}\bb_{\th}\, \text{d}\th \}\ , \\
  f_{\nn}\esnu(\xx,\yyag) &\eqd  \mu_{\nn}\esnu f\big( \txt\frac{1}{\mu_{\nn}\esnu} \xx ,\yyag ;\txt\frac{1}{\mu_{\nn}\esnu}\txt\int_{\Theta_{\nn}\esnu} \ss_{\th} \dth \big)  , \quad \forall (\xx,\yyag)\in \mu_i\esnu\M\times \M.
\end{align*}
Then, under \Cref{ass_X_nonat,ass_ut_nonat,ass_measur,assp:gradientContinuity},  $(\Gna\esnu(A))_{\snu}=\big( \I\esnu, \FX\esnu,A, (f\esnu_\i)_{\i\in \I\esnu} \big)_{\snu} $
 is a sequence of finite-type approximating games of the  nonatomic aggregative game $\Gna(A)$ with an aggregative constraint. 
\end{proposition}
\begin{proof} 
Let us show the three properties required by \Cref{def:approx_seq} as follows.

i) For each $\nn\in\I\esnu$, $ \X_{\nn}\esnu =\left\{\xx \in \rit^T  : \A\xx \leq \txt\frac{1}{\mu_{\nn}\esnu}\txt\int_{\Theta\esnu_{\nn}}\bb_{\th}\, \dth \right\}$. Then, by a result generalized from \cite[Thm. 4.1]{batson1987combinatorial}, 
 there is a constant $C_0$ such that, for each $\th'\in \Theta_{\nn}\esnu$: $d_H\left( \X_{\th'},  \X_{\nn}\esnu \right) \leq C_0 \norm{ \bb_{\th'}-  \frac{1}{\mu_{\nn}\esnu}\int_{\Theta\esnu_{\nn}}\bb_{\th}\, \text{d}\th }
\leq \frac{C_0}{\nu} \norm{ \bm{\ob} -\bm{\ub} }$. Hence, $\mdset\esnu$ tends to 0.  Note that in this case, the assumption of $\{\spa \X_\th \}_{\th\in\Theta}$ being finite is naturally satisfied.

ii) For each  $\nn \in \I\esnu$ and each $\th'\in \Theta_{\nn}\esnu$,  for all $(\xx,\yyag)\in \M^2$, one has:
\begin{align*}
 \norm{ \deriv_1 f\esnu_{\nn}  (\xx, \yyag) \ -  \ \deriv_1 f_{\th'}(\xx,\yyag) }
& =    \norm{  \deriv_1 f\big(  \xx, \yyag ; \txt\frac{1}{\mu_{\nn}\esnu}\txt\int_{\Theta_{\nn}\esnu} \ss_{\th} \dth\big) \ - \ \deriv_1 f(\xx,\yyag;\ss_{\th'}) }\\
 &\leq L_3 \|\frac{1}{\mu_{\nn}\esnu}\int_{\Theta_{\nn}\esnu} \ss_{\th} \dth - \ss_{\th'} \| \leq  \frac{L_3}{\snu} \norm{\bm{\os}-\bm{\us} } \ ,
\end{align*} 
by the Lipschitz continuity of $\deriv_1 f(\xx,\yyag;\cdot)$. Hence, $\mduti\esnu$ tends to 0.

iii) By definition,  $D\esnu=0$.
\end{proof}

\begin{remark}  
In \Cref{prop:approx_seq_finitedim}, instead of the average value of the characteristics of nonatomic players on $\Theta_{\nn}\esnu $, one can use the characteristic value of any nonatomic player in $\Theta_{\nn}\esnu$. 
\end{remark}

\begin{remark}
By construction, in both sequences above, the compacity and convexity of the feasibility sets $(\X_i)_i$ and the convexity and continuity of cost functions $(f_i)_i$ are naturally inherited from the properties assumed on $(\X_\th)_\th$ and $(f_\th)_\th$. 
This should often be the case when building a sequence of approximating games from a nonatomic game with an infinity of types.
\end{remark}

\section{Illustration on a Smart Grid Example} 
\label{sec:example_energy}

\newcommand{\Em}{E_{\mr{max}}}
\newcommand{\Et}{E_{\mr{tot}}}
\newcommand{\xto}{x_{\th,O}}
\newcommand{\xtp}{x_{\th,P}}
In this section the results are derived on a simple example for illustration, in the framework stated in introduction. In this example, we will be able to compute explicitly the aggregate equilibrium of the infinite-type nonatomic game.
\medskip
 
We suppose that the energy operator has access to the probability distribution of the amount of flexible energy in the $N=30$ millions French households: let us assume that this distribution is uniform on $[0, \Em]$ with $\Em=20$kWh (kiloWatthour), that is $\phi_E(E)=\frac{1}{\Em}$ for $E \in [0,\Em]$.

 Then the quantile function (or inverse cumulative distribution function), scaled by the population size, is given by $E_\th= F_E^{-1}(\th) = \th \Em N$, for each $\th \in \Theta=[0,1]$. In this case, the action set mapping $\X_.: \Theta \rightrightarrows \rit^2$ is given by:
 \begin{equation*}
 \forall \th \in \Theta, \ \X_\th = \left\{\xx_\th= (\xto,\xtp)\in\rit^2_+ \ | \ \xto+ \xtp = E_\th \right\}\ , 
 \end{equation*}
 which gives an infinity of different action sets. 
Let us consider, as said in the introduction, that there are two prices:
\begin{equation*}
c_O(\xxag)= \frac{a_O}{N} \xag_O \text{ and } c_P(\xxag)=\frac{a_P}{N} \xag_P
\end{equation*}  
for off peak and on peak periods, with $a_P>a_O$, that depend only the aggregate energy  on off peak period $X_O$ and on peak period $X_P$ (or rather on the average energy that consumers ask on these periods). 
Thus, the cost function of each player $\th$ is given, as in the example of public products game given by \eqref{eq:common_form_cost}:
\begin{equation*}
\ \forall\xx_\th \in \X_\th , \ f_\th(\xx_\th)= \xto \times  c_O(\xxag)+ \xtp  \times   c_P(\xxag) \ = \langle \xx_\th, \cc(\xxag) \rangle ,
\end{equation*}
where $\cc= (c_0, c_P)$.
Hence, all players have the same cost function: the infinite number of types is only due to the infinite number of different action sets.

Owing to to \Cref{prop:monotonemap}, the nonatomic game $G$ obtained is aggregatively strongly monotone with modulus $\beta=\frac{a_O}{N}$. However, the game is not strongly monotone.

It turns out that on this toy example, one can directly compute the aggregate profile of the VWE, as the IDVI of \Cref{def:ve-infinite} asks to find $\sxx \in \FX$ %
 such that:
\begin{align}
& \int_{\Theta} \langle \g_{\sxx} (\th), \xx_\th - \sxx_\th \rangle \dth \geq 0,\quad \forall \xx \in \FX %
 \nonumber \\
\Longleftrightarrow  \  &  \int_{\Theta} \langle \cc(\xxag^*), \xx_\th - \sxx_\th \rangle \dth \geq 0,\quad \forall \xx \in \FX %
 \nonumber \\
\Longleftrightarrow \   &  \langle \cc(\xxag^*), \xxag - \xxag^* \rangle \geq 0,\quad \forall \xxag \in \Sxag %
\ . \label{eq:VWEsimpl}
\end{align}
This simplification holds because, for each $\th$, $\g_{\sxx} (\th)$ depends only on the aggregate $\xxag^*$ (which would not be the case for general nonlinear cost functions). As a result the VI obtained is of finite dimension. 

In this example, the aggregate action set $\Sxag$  can also be characterized easily, although this would not be the case for arbitrary  sets $(\X_\th)_\th$. In fact, as the aggregate flexible energy available is $\Et \eqd
\int_{\Theta} E_\th \dth =\frac{1}{2}N\Em$, we obtain:
\begin{equation}
\Sxag = \left\{ (X_O, X_P) \in \rit^2_+ \ | \ X_O + X_P = \Et \right\} \ .
\end{equation}
Indeed, if $X_O + X_P=\Et $, then taking $\xx_\th= \xxag \frac{ E_\th }{\Et}$, we have  $\xx \in \FX$ and $\int_\Theta \xx_\th \dth= \xxag$. The converse inclusion is clear.  Consequently, we obtain from \eqref{eq:VWEsimpl} that $\xxag^*$ is the solution of the quadratic program:
 \begin{align*}
& \min_{\xxag} \tfrac{a_O}{N} \times  \tfrac{1}{2}\xag_O^2 + \tfrac{a_P}{N} \times \tfrac{1}{2} \xag_P^2 \\
 &\xag_O + \xag_P= \Et \\
&  0 \leq \xag_O, \xag_P %
 \end{align*}
 that is: $\xxag^*= (\xag_O^*,\xag_P^*)= (\tfrac{a_P}{a_O+a_P}\Et, \tfrac{a_O}{a_O+a_P} \Et)$.
 
 \medskip
 
 Now, let us define a sequence of finite-type approximating games $G\esnu$ to approximate $G$, with for each $\snu \in \nit^*$, $I\esnu= \snu $. Let us drop the index $\snu$ for simplicity in the remaining. Let us split up the population uniformly with $\Theta_i\esnu=[\tfrac{i-1}{I},\tfrac{i}{I}]$, for each $i\in\I=\{1, \dots ,I\}$. 
 
Because of the linearity of $\th \mapsto E_\th$, considering the uniform approximation detailed in \Cref{subsec:approx_uniform} case 1, one will obtain directly $\xxag^*$. 
For the example, let us rather consider the approximating games defined with, for each $i \in \I$:
\begin{equation} \X_i \eqd   \{\xx_i \in\rit^2_+ \ | \ x_{i,O}+ x_{i,P} = E_i \eqd \tfrac{i}{I}N \Em \} \ .
\end{equation} Besides, we naturally take $f_i\eqd f_\th$ for each $i$ (as the cost function is the same for each player).%

One can observe that we get for each $i$, $\dset_i=\frac{N \Em }{I}=\frac{2 \Et }{I} \rightarrow 0$, and of course $\duti_i=0$. 

On the other hand, computing the aggregate approximate equilibrium, similarly to  \eqref{eq:VWEsimpl}, one  obtains:
\begin{equation*}
\hxxag^I= \left(\tfrac{a_P}{a_O+a_P}\Et(1+\tfrac{1}{I}), \tfrac{a_O}{a_O+a_P} \Et(1+\tfrac{1}{I}) \right) = (1+\tfrac{1}{I}) \xxag^* \ ,
\end{equation*}
and thus we have:
\begin{equation}
\| \hxxag^I - \sxxag \| =\frac{ \|\sxxag \|}{I} =  \frac{\sqrt{ a_O^2+a_P^2} }{a_O+a_P} \Et \times \frac{1}{I} .
\end{equation}
However, from \Cref{thm:converge_with_u}, as we can compute 
$\Bdf= \max_{\xxag \in \Sxag} \norm{ \cc(\xxag) }=\frac{a_P}{N}\Et $, we obtain the more conservative convergence bound (we can replace $(4\Bdf +1)$ by $2\Bdf$ since $D\esnu= \mduti\esnu=0$):
\begin{align*}
&\| \hxxag^I - \sxxag \|^2 \leq \frac{1}{\beta} 2 \Bdf \mdset^I = \frac{N}{a_O} 2 \frac{a_P}{N} \Et \times \frac{2 \Et}{I} \\
\Longleftrightarrow \ & \| \hxxag^I - \sxxag \| \leq 2\Et  \sqrt{\frac{a_P}{a_O}} \times \frac{1}{ \sqrt{I}} \ . 
\end{align*}

\section{Conclusion}\label{sec:conclusion}

\Cref{thm:converge_with_u} provides a precise theoretical result for the use of symmetric, finite-dimensional, (variational) Wardrop equilibria (S(V)WE) as an approximation of the (V)WE in a strongly monotone or aggregatively strongly monotone nonatomic aggregative game with an infinity of players types, with or without aggregative constraints. There are numerous research themes related to this result and our topic in general.

First, one needs to find efficient algorithms for the computation of finite dimensional variational inequalities arising as the characterization of SVWE. An extensive literature exists in this regard but our particular case of aggregative game with aggregative constraints may lead to special methods or improvements on existing results \cite{gramm2017}.

Then, the extension of evolutionary dynamics for population games and the related algorithms, to the framework of nonatomic games with infinitely many classes of players can be non trivial.
 A recent work \cite{Hadikhanloo2017} proposes online learning methods for population games with heterogeneous convex action sets.
 The presence of aggregate constraints adds two additional difficulties for considering %
 evolutionary dynamics in population games, as those dynamics are based on unilateral adaptations from players. On the one hand, in the presence of coupling constraints,  unilateral deviations by players may well lead to an action profile violating the coupling constraint. On the other hand, a feasible deviation in the action profile cannot always be decomposed into unilateral deviations of players. 

Last, our results are limited to monotone games and the convergence result is limited to strongly monotone games. The study of nonatomic aggregative games that are not monotone needs probably other approaches. Indeed, even for population games where there are only finitely many types of players, there exist much fewer results for games that are not linear, potential or monotone. 
The question of whether or not it is possible to obtain similar convergence results as those stated in \Cref{thm:converge_with_u}  without monotonicity assumptions constitutes an interesting path for future work.

\section*{Acknowledgments}
\noindent We thank St\'ephane Gaubert, Sylvain Sorin,  Marco Mazzola, Olivier Beaude and Nadia Oudjane for their insightful comments.

\newpage

\begin{appendix}

\section{Extension of results to the subdifferentiable case} \label{app:nonsmooth}

In this section, we explain briefly how our results extend to the case of convex nonsmooth cost functions, considering subdifferential instead of gradients of convex costs. 

The essence of the proofs are roughly the same as in the smooth case, but considering subdifferentials requires some additional technical arguments. The full proofs can be found in \cite{jacquot2018nonsmooth}. The authors decided to formulate the results in the smooth case so that the key arguments and ideas appear clearly.

\medskip

Recall that the  \emph{subdifferential}, i.e. set of \emph{subgradients} of a convex function $f$ at $\xx\in \rit^T$ in its domain $C$, which is a convex set in $\rit^T$, is denoted by $\partial f(\xx)$. 
Recall that  $\g\in \rit^{T}$ is a subgradient of $f$ at $\xx$, denoted $\g \in \partial f(\xx) $, iff for all $\zz\in C$, $f(\zz) \geq f(\xx) + \langle \g, \zz-\xx\rangle$.

\bigskip
One has to consider the correspondence of subdifferential  $\H: L^2([0,1],\M) \rightrightarrows L^2([0,1],\rit^T)$, which associate to each profile $\xx \in L^2([0,1],\M)$ and each player $\th$ the set of subgradients of her cost functions: 
\begin{equation}\label{eq:correspNonatom}
\H(\xx) \eqd \{ \g= (\g_\th)_{\th\in\Theta}\  | \ \g_\th \in \partial_1 f_\th(\xx_\th, \txt\int \xx), \   \forall a.e. {\th\in \Theta}  \}  \, ,\quad \forall \xx\in L^2([0,1],\M) .
\end{equation}
In other words, $\H(\xx)$ is the collection of measurable (and integrable because of \Cref{ass_ut_nonat}.iii adapted to assume uniform boundedness of $\partial_1 f_\th$ ) selections of a subgradient for each $\xx_\th$.
Most of the paper can be interpreted in the nonsmooth framework by 
\begin{itemize}[wide]
\item replacing $\g_{\xx}$ by an element of $\H(\xx)$ in the equations, 
\item considering the Hausdorff distance $d_H$ between subdifferentials instead of the Euclidean distance between two gradients (e.g. for $\duti$ in  \Cref{def:approx_seq} of a sequence of finite-type approximating games),
\item a direct implication is that we have to consider \emph{generalized } variational inequalities (GVI), finite or infinite-dimensional, instead of VIs.

 For instance, $\langle \g_{\sxx} , \xx^* - \xx \rangle \leq 0 , \forall \xx \in \FX$ becomes $\exists \g \in \H(\sxx) , \ \langle \g , \xx^* - \xx \rangle \leq 0 , \forall \xx \in \FX$.
\end{itemize}
It is useful to introduce the best-reply correspondence $Br: \Sxag \rightrightarrows \FX$:%
\begin{equation*}
Br(\yyag) \eqd \{ \xx\in \FX : \xx_\th \in \txt\arg\min_{\X_\th} f_\th(\cdot, \yyag), \forall \th\in \Theta\}, \quad \forall\, \yyag\in \Sxag \ , 
\end{equation*}
and, for $\yyag\in \Sxag$ and $\xx\in Br(\yyag)$, the correspondence $\D({\xx,\yyag})$: $\Theta \rightarrow \rit^T$ defined by:
\begin{equation*}
\D({\xx,\yyag}) (\th)\eqd \{ \g_\th \in \partial_1 f_\th(\xx_\th,\yyag) \ | \  \langle \g_\th , \zz_\th - \xx_\th \rangle \geq 0\, , \quad\forall \zz_\th \in \X_\th \} \ \,\  \forall \th\in \Theta
\end{equation*}
which is nonempty (by first order conditions) and closed-valued. To get similar results as in the smooth case, we need to make the following additional assumption:
  \begin{assumption}\label{assp:measurSubdiffCorresp}
 For all $\yyag\in \Sxag$ and all $\xx\in Br(\yyag)$, $\D({\xx,\yyag})$ is a measurable correspondence. 
 \end{assumption}
 One can show that $\H(.)$ and $Br$ have  nonempty values. Then, instead of \Cref{lm:bestreply}, we use the compact-valued selection theorem  \cite{aumann1976integration} to obtain for each $\xx\in Br(\yyag)$ the existence of a measurable mapping $\th \mapsto \g_{\xx}(\th)$ such that $\forall \th\in \Theta, \g_{\xx} (\th) \in \D({\xx,\yyag})(\th)$.
 
 We can then obtain a characterization of WE similar to \Cref{thm:agg_wardrop}, where the differential $\g_{\xx^*}$ is replaced by the existence of an element in $\H(\xx^*)$. 
 The existence result in \Cref{thm:exist_we} is also valid in the subdifferentiable case.
 
 \medskip
 
The monotonicity of $G$ is defined as in \Cref{def:mono_non}, where the inequalities on $\g_\x$ and $\g_\y$ now have to hold for each pair of elements of the correspondences $(\g_{\xx},\g_{\yy}) \in \H(\xx) \times \H(\yy)$.
Properties  characterizing monotonicity given in \Cref{prop:monotonemap} follow with essentially the same proof, having in mind that for any $\th$, $\partial_1 f_\th(\xx_\th, \yyag)=\{\cc(\yyag)+\g:\g\in \partial (-u_\th)(\xx_\th) \}$.

  \medskip
  
In presence of coupling constraints, a VWE (\Cref{def:ve-infinite}) is also defined by $\xx^* \in \FX(A)$ and the existence of an element $\g\in\H(\x^*)$ satisfying the infinite dimensional GVI  \eqref{cond:ind_opt_ve_inf}. 

The existence of a VWE (similar to \Cref{th:exist_ve_inf}) can also be obtained in the nonsmooth case, where the continuity of the gradient is replaced by upper-hemicontinuity of the correspondence $\H$ and applying results of \cite[Corollary 2.1]{ding1996monotoneGVI}.
 The uniqueness conditions associated to monotonicity detailed in \Cref{th:unique_vwe} follow as well with essentially the same proof as for the smooth case.%
 
\medskip

The main result, \Cref{thm:converge_with_u}, is obtained for the nonsmooth case with the same bounds on the convergence rate.

\newpage

\section{On the Relationship between Nash and Wardrop Equilibria}
\label{app:nash-wardrop}
The objective of this paper is to approximate the equilibrium of a nonatomic game with an infinity of players types, by considering approximating games with a finite number of players types.

A natural idea would also be to consider approximating games with a  finite number of \emph{players}, this number of players growing to infinity to approximate the nonatomic population game. 

Indeed, we briefly explain in this appendix how we can obtain similar convergence results by adopting this approach considering finite-player atomic games. 
The approach is fully developed in \cite{jacquot2018nonsmooth}.

\bigskip

The main difference and difficulty under this approach is that the equilibrium concept to consider for finite-player games is no longer  Wardrop Equilibrium, but Nash Equilibrium (NE). 
As the number of players is finite, an individual action $\xx_i
$ of a player $i$ does have an impact on the aggregate action $\xxag= \sum_i \xx_i$. 

As a result, the modified cost function $\hf_\i:(\xx_\i, \xxag_{-\i})\mapsto f_i(\xx_i,\xxag_{-i}+\xx_i)$ naturally appears, where the action of $\i$ is taken into account in the aggregate action.

This modified cost function, and the impact of individual actions in general, have to be considered  both in the assumptions and in the definitions of the different concepts used in this paper.

Nash equilibrium are naturally characterized by finite-dimensional variational inequalities under convexity hypotheses: in the atomic case, we need the additional following assumption:
\begin{assumption}\label{ass:convex_NE}
For an atomic game  $\GA=(\I, (f_i)_i,(\X_i)_i,A)$ with a finite set of players $\I$ and cost functions $(f_i)_i$, the associated functions $\big(\hf_i(.,\xxag)\big)_i$ are convex. 
\end{assumption}
Note that this convexity is not necessarily implied by the convexity of $f_i(.,\xxag)$. Under this additional assumption, we obtain a GVI (where G stands for \emph{generalized} in the nonsmooth case, see \ref{app:nonsmooth}) characterization of NE, similar to the one for SVWE given in \Cref{prop:SVWEfiniteChar}, and an existence result:

\begin{definition}[Variational Nash Equilibrium (VNE), \cite{harker1991gne}]\label{def:ve-finite-nonsmooth} A (variational) Nash equilibrium of  atomic game $\GA$  is a solution to the following GVI problem: 
 \begin{equation}\label{cond:ind_opt_ve-nonsmooth}
\text{Find } \hxx\in \FX(A) \text{  s.t. } \exists\, \g\in H(\hxx)  \text{ s.t. }   \txt \big\langle \g, \xx- \hxx\big\rangle\geq 0,\; \forall \xx \in \FX(A).
 \end{equation}
 where  the subgradients correspondence $H:\FX\rightrightarrows \rit^{IT}$ is given as:
\begin{equation*}
\forall \xx\in \FX, \; H(\xx) \eqd\{(\g_i)_{i\in \I}\in \rit^{IT}: \g_i \in \partial_1 \hf_\i(\xx_\i, \xxag_{-\i}) , \  \forall i\in \I\} = \prod_{i\in\I} \partial_1 \hf_\i(\xx_\i, \xxag_{-\i}) \ .
\end{equation*}
 In particular, if  $\Sxag\subset A$,  a VNE is a NE .
\end{definition}

\begin{proposition}[Existence of VNE]\label{prop:exist_ve-nonsmooth}
  Under \Cref{ass_X_nonat,ass_ut_nonat} (compacity and convexity) on $(\X_i)_i$ and \Cref{ass:convex_NE}, the atomic game $\GA=(\I, (f_i)_i,(\X_i)_i,A)$ admits a VNE. 
\end{proposition}

To obtain a convergence result of the (V)NEs in a sequence of atomic games, we  need some stronger properties than for the sequence of finite-type approximating games.  In addition to  \Cref{def:approx_seq}, we assume that:
\begin{enumerate}
\item the number of players tends to infinity:  $I\esnu \underset{\snu \rightarrow \infty}{\longrightarrow}  \infty$; \
\item each player becomes infinitesimal: $\mu_i= \mu(\Theta_i)\underset{\snu \rightarrow \infty}{\longrightarrow}  0 $ ;\
\item in the gradient (or subdifferential) of a player, the impact of her own action on the aggregate profile vanishes along the sequence, by considering the additional parameter (given in the subdifferential case):
\begin{equation}\label{eq:def_ld}
\ld\esnu_\i \eqd \sup_{(\xx,\yyag)\in \M^2} \sup_{\g\in \partial_1 \hf\esnu_\i(\mu\esnu_\i\xx, \yyag-\mu\esnu_\i\xx)}d \left(\g, \partial_1 f\esnu_\i(\mu\esnu_\i\xx, \yyag)\right)\ \underset{\snu \rightarrow \infty}{\longrightarrow} 0
\end{equation}
\end{enumerate}

\medskip
Then one obtains similar convergence result as \Cref{thm:converge_with_u}, with the only difference being that:
\begin{itemize}
\item in the upper bound, $\mduti\esnu$ is replaced by $(\mduti\esnu + \mld\esnu)$,  where $\mld\esnu \eqd \max_i \ld\esnu_i)$,
\item for the convergence in $\L_2([0,1],\M)$, one has to consider a projection of the (V)NE on $\L_2([0,1],\M)$, where each  $\theta \in \Theta_i$ is associated to the action of $i\in\I$.
\end{itemize}

\bigskip

Note that this approach has also another interest, as the convergence theorem that we obtain in this case can be interpreted in the reverse way: under the right assumptions,  in a sequence of atomic games converging (in the sense given by the definition of  finite-type approximating games) to a nonatomic aggregative game, the sequence of Nash equilibria converge to an equilibrium (the SVWE) of this limit nonatomic game.

In some particular cases where the continuous SVWE can be computed explicitly as a function $\xx:\Theta \rightarrow \M$ (for instance in the example derived in \Cref{sec:example_energy}), this would give an approximation of a Nash equilibrium in the atomic  (reality) aggregative game, as the finite number of atomic players is very large.
\end{appendix}

\bibliographystyle{ims}

\bibliography{../../../bib/longJournalNames,../../../bib/biblio1,../../../bib/biblio2,../../../bib/biblio3,../../../bib/biblioBooks}

\end{document}